\documentclass{article}

\usepackage{amsfonts}
\usepackage{color}
\usepackage{verbatim}   
\usepackage{hyperref}   
\usepackage{enumerate}
\usepackage{amssymb}
\usepackage{amsmath,amsthm}
\usepackage{graphicx}

\usepackage{tikz}
\usetikzlibrary{matrix,arrows}

\newtheorem{theorem}{Theorem}

\newtheorem{corollary}[theorem]{Corollary}

\newtheorem{example}[theorem]{Example}

\newtheorem{lemma}[theorem]{Lemma}

\newtheorem{proposition}[theorem]{Proposition}

\newtheorem{remark}[theorem]{Remark}

\newcommand{\Q}{\mathbb Q}
\newcommand{\C}{\mathbb C}

\newcommand{\HH}{\mathbb H}

\newcommand{\Hom}{\mathrm{Hom}}

\newcommand{\aut}{\mathrm{Aut}}
\newcommand{\com}{\mathrm{Comm}}

\newcommand{\Id}{\mathrm{Id}}

\newcommand{\sym}{\mathrm{Symm}}

\newcommand{\Gal}{\mathrm{Gal}}

\input xy
\xyoption{all}

\title{  Galois action on   Fuchsian surface  groups and their  solenoids
 \thanks{  Second author partially supported by Spanish Goverment Research Project MTM2016-79497-P 
\newline AMS Subject Classification: 14H30, 20H10, 30F10, 11F06, 22D99
\newline Key words: Algebraic curves, Arithmetic Fuchsian groups,  Commensurators, Galois invariants, Group completions, Riemann surface solenoids. }}

\author{  Amir D\v{z}ambi\'c \\
Christian-Albrechts-Universit\"{a}t zu Kiel \\
\emph{ dzambic@math.uni-kiel.de } \\ 
\\  
   Gabino Gonz\'alez-Diez  \\ 
 Universidad Aut\'onoma de Madrid \\
\emph{   gabino.gonzalez@uam.es} }

\date{}

\begin{document}

\maketitle

 \begin{abstract}
  Let $C$ be a  complex algebraic curve  uniformised by a Fuchsian group $\Gamma$.
In the first part of this paper
 we identify the automorphism group of the solenoid associated with $\Gamma$ with the
	 Belyaev completion of  its commensurator  $\com(\Gamma)$ and we use this identification to show that the isomorphism class of this completion is an invariant of the natural Galois action of $Gal(\C/\Q)$ on algebraic curves.  In turn this fact yields a proof of the Galois invariance  of the arithmeticity of $\Gamma$ independent of Kazhhdan's.

	In the second part  we focus on  the case in which $\Gamma$ is arithmetic.
 The list of  further Galois invariants we find includes: i)  
 the periods of $\com(\Gamma)$, ii) the  solvability of the  equations 
	$X^2+\sin^2 \frac{2\pi}{2k+1}$ 
	in the invariant quaternion algebra of $\Gamma$ and iii) the property of   $\Gamma$ being a   congruence 
subgroup.

 \end{abstract}

\section{Introduction and statement of results}
  We recall that 
  two subgroups $H_1$ and $H_2$   of a group $G$ are said to be \textit{commensurable}
if $H_1 \cap H_2$ has finite index in both $H_1$ and $H_2$ and that 
the  \emph{commensurator} of a subgroup  $H<G$ is the subgroup 
$\com(H)$ consisting of the elements $g\in G$ such that $H$ and $g H g^{-1}$ are commensurable.

Let $C$ be a compact Riemann surface (or, equivalently, a nonsingular  complex projective curve)  of genus $g\geq 2$ and let
 $ \Gamma\cong \pi_1(C)$  be the Fuchsian surface group uniformising $C.$  Here we  deal with 
the  commensurator of $\Gamma$  in  $PSL_2(\mathbb{R})$, namely
$$
\com(\Gamma) = \{\alpha \in PSL_2(\mathbb{R}): \alpha \Gamma \alpha^{-1}\cap \Gamma \text{ has finite index in both } \Gamma \text{ and }  \alpha \Gamma \alpha^{-1}\}
$$
We observe that the group $\Gamma$
 is determined by $C$ only up to conjugation in  $PSL_2(\mathbb{R})$ and that commensurable groups have the same commensurator.  Another  object which  depends only on the  commensurability class of $\Gamma$ is the  Sullivan \emph{solenoid}
 $\mathcal{H}_{C}$ (or $\mathcal{H}_{\Gamma}$) associated with $C$ (or with  $\Gamma$).
 These two objects will play a central role in this paper. They are related by the fact  that the 
 automorphism group of the solenoid, let us denote it  $\aut(\mathcal{H}_{C})$ (or $\aut(\mathcal{H}_{\Gamma})$),   is generated by  $\com(\Gamma)$
and $\widehat{\Gamma}$, the profinite completion of $\Gamma.$

Let $Gal( \C/\Q)$  denote the group of field automorphisms of $\C$, and let
 $\sigma \in Gal( \C/\Q)$.
 The obvious (Galois) action of $Gal( \C/\Q)$ on the coefficients of the defining equations
  transforms a complex algebraic variety $X$ 
	 defined over a subfield $k\subset \C$
into another algebraic variety $X^{\sigma}$ defined over the subfield $k^{\sigma}:=\sigma(k)$.
	
	Finding invariants for this action is an important problem in  Complex Algebraic Geometry. For instance, the Betti numbers and the profinite completion of the fundamental group are Galois invariants but, in dimension $\geq 2$,  the fundamental group itself and the universal cover are not (see e.g. \cite{Se}, \cite{milnesuh}, \cite{Ca}, \cite{GJ}  \cite{GJT}, \cite{GR1}, \cite{GR2} and \cite{Go1}). Such phenomena do not occur in dimension one 
since in this case the Galois action is genus preserving, but Grothendieck's theory of \emph{dessins d'enfants} and algebraic curves share similar ideas and goals (\cite{Gro}, see also \cite{GG} and \cite{JW}).
 
	 Let  $\sigma \in Gal( \C/\Q)$ and let
	$\Gamma^{\sigma}$ denote the Fuchsian group
  uniformising the conjugate algebraic curve 
	$C^{\sigma}$. 
	 We shall refer to $\Gamma^{\sigma}$ as the Galois conjugate of $\Gamma$ by $\sigma$. We observe that the rule 
	$(\sigma, \Gamma) \to \Gamma^{\sigma}$ defines an action of  $Gal( \C/\Q)$ 
	on the set of $PSL_2(\mathbb{R})$-conjugacy classes of Fuchsian surface groups.
	As it has been said, in   this ($1$-dimensional) case
  $\Gamma$  and  $\Gamma^{\sigma}$    must be isomorphic as abstract groups. However, little seems to be known about the relationship between   $\Gamma$  and $\Gamma^{\sigma}$ as subgroups of $PSL_2(\mathbb{R})$. Note that  $\com(\Gamma)$ depends precisely on the way 
      $\Gamma$ sits inside $PSL_2(\mathbb{R})$.

		In the first part of this article we study
	the group $\aut(\mathcal{H}_{\Gamma})$; we  show  it to be the Belyaev completion of $\com(\Gamma)$ (see the definition in \ref{Belyaev}) and  derive some properties  which remain invariant under Galois action. In the second part we also look for Galois invariants  but we use mainly tools pertaining to the theory of arithmetic groups.

	\subsection{The automorphism group of a solenoid  $\aut(\mathcal{H}_{C})$.}
	
	The explicit description of the automorphisms of $\mathcal{H}_{C}$ was made by Odden in \cite{Odd},   where he showed that  any element $F\in \aut(\mathcal{H}_{C})$ can be written as a product $F=\alpha \tau$, with $\alpha \in \com(\Gamma)$ and $\tau \in \widehat{\Gamma}$ essentially in a unique way.
	 However the author remarks that \textit{``(this) theorem 
does not shed light on its group structure"}.  

We will consider the completion  of $ \com(\Gamma)$   relative to the topology determined by the finite index subgroups of  $\Gamma$  and, 
following Belyaev \cite{Be}, we will make it into  a totally disconected locally compact 
	group which we will denote   $\overline{\com(\Gamma)}$. 
	 	Then  we will  prove (Theorem \ref{Isom}) 
  \begin{itemize}
  \item   $ \aut(\mathcal{H}_{C}) \cong  \overline{\com(\Gamma)}$  \hspace{1.6cm} (isomorphism of  groups)
   \end{itemize}

	Then we will show that that there is a 
   natural action of    $Gal( \C/\Q)$ on $ \mathcal{H}_{C}$ which, for any $\sigma \in Gal( \C/\Q)$,  induces the following  isomorphisms (Theorem \ref{mainth}):

	\begin{itemize}
	 \item \  \  \  $\overline{\com(\Gamma)} \simeq \overline{\com(\Gamma^{\sigma})} \hspace{1.5cm} \text{ (isomorphism of   topological groups}) $ 
\item \ \ \  $ \dfrac{ \com(\Gamma)}{\Gamma} \simeq \dfrac{ \com(\Gamma^{\sigma})}{\Gamma^{\sigma}} \hspace{1.3cm} \text{ (isomorphism of sets of co-sets)} $	
	\end{itemize}	
		 In view of a well-known theorem of Margulis \cite{Mar}, which states that $\Gamma$ is arithmetic if and only if $[ \com(\Gamma):  \Gamma]=\infty$, 
		the second of these isomorphisms provides an alternative proof of the following result which  Kazhhdan (\cite{Kaz}) proved in arbitrary dimension.

\begin{itemize}
  \item  $C$ is uniformised by an arithmetic group if and only if $ C^{\sigma}$ is.
  (Corollary \ref{kaz}).
   \end{itemize}
				 
In the particular (but generic) case in which  $\Gamma$ is non-arithmetic 
the  isomorphism class of the  group  $\com(\Gamma)$ is also   a Galois invariant; more precisely we will prove the following result (Theorem \ref{ThNonArit})
		\begin{itemize}
	 \item Suppose that  $\Gamma$ is non-arithmetic, then 
	  $\com(\Gamma) \simeq \com(\Gamma^{\sigma}) .$
	              
	\end{itemize} 
 The second part of the article will be devoted to the arithmetic case.
 
	\subsection{The arithmetic case}
	 The main tool to understand the relation between $\Gamma$ and $\Gamma^{\sigma}$ in the arithmetic case is a theorem by Doi-Naganuma. This result establishes the relation between the quaternion algebra  $(k\Gamma, A\Gamma)$ associated with  $\Gamma$  and the quaternion algebra 
 $(k\Gamma^{\sigma}, A\Gamma^{\sigma})$ associated with the Galois conjugate group  $\Gamma^{\sigma}$ 
	(see Theorem \ref{mainI} for the precise statement and \ref{arithmetic} as well as \ref{quatalgnf} for main notions from the theory of arithmetic Fuchsian groups and quaternion algebras). For instance this theorem states that $k\Gamma^{\sigma}=(k\Gamma)^{\sigma}$, hence    if $\Gamma$ and $\Gamma'$ are two arithmetic Fuchsian groups whose invariant trace fields are not Galois conjugate, then the curves $C=\mathbb H/\Gamma$ and $C'=\mathbb H/\Gamma'$ cannot be Galois conjugate.
 
In this paper we will  use this theorem to study the behaviour of the torsion of the group  $\com(\Gamma)$ under Galois action. To be more precise, 
	  let $\mathcal{P}(\Gamma)\subset \mathbb{N}$ denote the set of orders (or periods)  of the finite order  elements of   $\com(\Gamma)$. Then we will show that (Theorem \ref{periods}).
	\begin{itemize}		
	\item	$\mathcal{P}(\Gamma) = 
	\mathcal{P}(\Gamma^{\sigma})$, for any $\sigma \in Gal( \C/\Q)$.	
		\end{itemize}	
		In other words the set of periods of  $\com(\Gamma)$ is  another Galois invariant
		 which could tell apart   surface groups uniformising curves in different Galois orbits.


	C. Maclachlan \cite{Mac1} has described the set $\mathcal{P}(\Gamma)\subset \mathbb{N}$ in terms of     the   invariant quaternion algebra 
		of $\Gamma,$  namely 
		$$\mathcal{P}(\Gamma)=\{m\in \mathbb{N}:
		\cos \frac{2\pi}{m}\in k\Gamma \text{ and the field } k\Gamma(e^{2\pi i/m}) \text{  embeds in }   A\Gamma\}
		$$ 		
	
	Using this  one can deduce, for instance, that  the subset $\mathcal{P}^{odd}(\Gamma)$ of odd periods admits the following simpler description 
(Proposition \ref{InvarAlgebr}).
		\begin{itemize}
	 \item $\mathcal{P}^{odd}(\Gamma)=\{m\in \mathbb{N}:
		 m \text{ odd and } A\Gamma \text{ contains
a square root of } -\sin^2 \frac{2\pi}{m} \}
		$ 
Hence the solvability of the quadratic equations 
	$X^2+\sin^2 \frac{2\pi}{2k+1}, \ \  k \in \mathbb{N}$ 
	is also Galois invariant.
		\end{itemize}
	Using this last description of the periods it will not be difficult to 	
	 construct  an explicit family of arithmetic groups $\Gamma_p, p \text{ an odd prime, }$ 
 enjoying the property that $\com(\Gamma_p)$ contains an element of odd prime order $q$ if and only if $p=q$ (Example \ref{Noncommens}).

Since algebraic curves uniformised by arithmetic groups are defined over number fields (Proposition \ref{DefinedOverNumberField}), in the arithmetic case   the  invariants of the action of $Gal(\mathbb{C}/\mathbb{Q})$ can be seen as 
invariants of the action of the more interesting group
$Gal(\overline{\mathbb{Q}}/\mathbb{Q})$, 
\emph{the absolute Galois group}. We will prove the  following result
	(Theorem \ref{AbsGalGroup})
	\begin{itemize}
	 \item   $Gal(\overline{\mathbb{Q}}/\mathbb{Q})$ acts faithfully on 
	the set of $PSL_2(\mathbb{R})$-conjugacy classes of Fuchsian surface groups. Moreover, 
	this action possesses the following invariants:
	 \begin{enumerate}
	\item The isomorphism class of the group $\overline{\com(\Gamma)}$. 
	\item The set $\mathcal{P}(\Gamma)$ of periods of 
	$\com(\Gamma).$ 
	\item The solvability of the quadratic equations 
	$X^2+\sin^2 \frac{2\pi}{2k+1}, \ \  k \in \mathbb{N}$ 
	\newline
	in the invariant quaternion algebra $A\Gamma$.
	\item  The Galois conjugacy class of the field $k\Gamma$. (In fact $k\Gamma^{\sigma}=(k\Gamma)^{\sigma}$ for any $\sigma \in Gal(\overline{\mathbb{Q}}/\mathbb{Q})$).
	\item  The property of being a conguence subgroup.
  \end{enumerate}
 \end{itemize}
  
\section{ Preliminaries} \label{solenoids}

\subsection{ Arithmetic Fuchsian groups} \label{arithmetic}

Here we give a brief summary of  some aspects  of the theory of arithmetic Fuchsian groups   
relevant to this paper. For details we refer to  
\cite{MR}, chapter 8.2  (but see also the nice survey article \cite{Mac2}). 
 
Let $k$ be a commutative field of characteristic $\neq 2$ and let $a,b\in k^{\ast}$, the group of units in $k$. We consider a \textit{quaternion algebra} $A$, also denoted by $(k,A)$, over $k$ identified with \textit{Hilbert symbol} $A=\big(\dfrac{a,b}{k}\big)$. This is the $k$-algebra which as a vector space over $k$ has a basis  
  $\{\vec{1},\vec{i}, \vec{j}, \vec{k} \}$ such that its   ring structure is determined by the relations $\vec{1}=1$,
 $\vec{i}^2=a$, 
$\vec{j}^2=b$ and $\vec{k}=\vec{i}\vec{j}=-\vec{j}\vec{i}$. For instance,   $\big(\dfrac{-1,-1}{\mathbb{R}}\big)$ is the classical Hamilton quaternion field $\mathbf{H}$ whereas $\big(\dfrac{1,1}{\mathbb{R}}\big)=M_2(\mathbb{R})$    \Big(with basis $ \vec{1}=  \left( \begin{array}{cc}
1 & 0  \\
0 & 1  \end{array}\right)$,  $ \vec{i}= \left( \begin{array}{cc}
1 & 0 \\
0 & -1 \end{array} \right)$,   $ \vec{j}= \left( \begin{array}{cc}
0 & 1  \\
1 & 0  \end{array} \right)$,
$ \vec{k}= \left( \begin{array}{cc}
0 & 1  \\
-1 & 0  \end{array} \right)$ \Big). 
In fact these are the only two quaternion algebras over the field of real numbers up to isomorhism, for it is known that
$\big(\dfrac{a,b}{\mathbb{R}}\big)$ equals  $\mathbf{H}$  if $a$ and $b$ are simultaneously negative and  equals $M_2(\mathbb{R})$ otherwise.
 The \textit{norm} of an element $x=x_0 +x_1\vec{i}+x_2\vec{j}+x_3\vec{k} \in A=\big(\dfrac{a,b}{k}\big)$ is defined to be 
$n_A(x)=x_0^2-ax_1^2-bx_2^2+abx_3^2$; for instance, when $A=M_2(\mathbb{R})$, $n_A(x)=det(x).$
 
\

From now on we assume that $k$ is a number field, that is a finite extension of $\mathbb{Q}.$   Let us denote by $R_k$ the ring of integers of    $k.$ 
An \textit{order}  
in $A$ is a subring $\mathcal{O}\subset A$ with $1$ which is a finitely generated 
$R_k$-module satisfying $\mathcal{O} \otimes_{R_k}k=A.$ Given an order  $\mathcal{O}\subset A$ one may consider the group 
$$\mathcal{O}^1=\{x\in \mathcal{O}:n_A(x)=1\}. $$
  We will call this group the \textit{norm-1-group of $\mathcal{O}$}. For instance, if $A= \big(\dfrac{1,1}{\mathbb{Q}}\big)\cong M_2(\mathbb{Q})$, then $\mathcal{O}= M_2(\mathbb{Z})$ is an order such that $\mathcal{O}^1\cong SL_2(\mathbb{Z})$. \newline
If $\tau:k \to  \mathbb{C}$  is a field  embedding  we will denote by $k^{\tau}$ the image of $k$  in $\mathbb{C}$ and by $A^{\tau}$ the quaternion algebra
$A^{\tau}:=\left(\dfrac{\tau(a),\tau(b)}{k^{\tau}}\right).$

 Let us now assume that $A=\big(\dfrac{a,b}{k}\big)$ satisfies the following conditions:
\begin{enumerate}
\item $k$ is a \textit{totally real number field} (which means that  the image $k^{\tau}$ of any embedding $\tau:k \to  \mathbb{C}$ lies in 
$ \mathbb{R}$)
\item There is an isomorphism $\rho:A\otimes_k \mathbb{R}=\big(\dfrac{a,b}{\mathbb{R}}\big) \xrightarrow{\sim} M_2(\mathbb{R})$.  One then says that $A$ is \textit{unramified} at the infinite place corresponding to identity embedding $id$.
\item $A^{\tau}\otimes_{k^{\tau}} \mathbb{R}:=\big(\dfrac{\tau(a),\tau(b)}{\mathbb{R}}\big) \cong \mathbf{H}$, for every $\tau$ different from the inclusion map. In this case one says that $A$ is \textit{ramified} at the infinite places $\tau\neq id$ (see \ref{quatalgnf}).
\end{enumerate}
In that situation a deep theorem by Borel and Harish-Chandra ensures that if $G<PSL_2(\mathbb{R})$ is a subgroup  commensurable to a group of the form $P( \rho(\mathcal{O}^1))$, where 
$\mathcal{O}$ is an order of $A$ and 
$P:SL_2(\mathbb{R}) \to PSL_2(\mathbb{R})$
stands for the canonical projection, then, 
$G$ is a \textit{finite volume} Fuchsian group. Groups arising in this way (possibly after conjugation in  $PSL_2(\mathbb{R}))$
are called \textit{arithmetic (Fuchsian) groups}  
(associated with the quaternion algebra $(k,A)$).   Note that the condition that $A$ is unramified at the identity and ramified at all embeddings different from the identity is a choice that is not essential for the validity of Borel-Harish-Chandra theorem. It is rather that any quaternion algebra $(k,A)$ with $k$ totally real and the property that there exists exactly one embedding $\tau_0$ such that 
$A^{\tau_0}  \otimes_{k^{\tau_0}} \mathbb{R}  \cong M_2(\mathbb R)$ and 
 $A^{\tau}  \otimes_{k^{\tau}} \mathbb{R}\cong \mathbf H$ for all embeddings $\tau\neq \tau_0$ will give rise to Fuchsian groups. The condition is often expressed as
\begin{equation}\label{tensor}
A\otimes_{\mathbb Q} \mathbb R\cong M_2(\mathbb R)\times \mathbf H^{[k:\mathbb Q]-1},
\end{equation}
where $[k:\mathbb Q]$ denotes the degree of the number field $k$.
 It is known that an arithmetic group  associated with $(k,A)$ fails to be \textit{co-compact} only when $(k,A)\cong(\mathbb{Q},M_2(\mathbb{Q} )).$

	The quaternion algebra $(k,A)$ with which a given arithmetic Fuchsian group $\Gamma$ is associated can be recovered as follows:
  
$$k= k\Gamma = \mathbb{Q } \left(trace(\gamma):  P(\gamma) \in \Gamma^{(2)}\right),$$ where $\Gamma^{(2)}$ is the group generated by all squares of elements of $\Gamma$, and
 
$$A= A\Gamma = \{ \ \textstyle \sum a_i \gamma_i: a_i\in  k\Gamma,  P(\gamma_i)\in \Gamma^{(2)} \  \}$$
 
Since the arithmetic group associated with a given quaternion algebra is only defined up to commensurability equivalence one finds that two arithmetic Fuchsian groups  $\Gamma_1$  and $\Gamma_2$  are commensurable if and only if 
$(k\Gamma_1, A\Gamma_1) = (k\Gamma_2, A\Gamma_2)$. Accordingly 
 $k\Gamma$ and  $A\Gamma$ are referred to as the
 \textit{invariant trace field} and the \textit{invariant quaternion algebra} respectively.

The commensurator of an arithmetic Fuchsian group $\Gamma$ can be obtained from its invariant quaternion algebra as 
$$
\com(\Gamma)= A^+/k^{\ast}\cong \widetilde{P}(A^+), 
$$
 where   $ A^+=  \{X\in A\Gamma: =n(X^{\tau})>0\ \text{for all}\ \tau\in \Hom(k,\mathbb R) \} $  is the group of all elements in $A$ with totally positive norm   and 
$\widetilde{P}$ stands for 
the homomorphism obtained by first embedding $A$ into $M_2(\mathbb R)$ via the identity embedding, then dividing each matrix $X$ by the square root of its determinant to get an element of $SL_2(\mathbb{R})$ and finally applying the projection $P:SL_2(\mathbb{R}) \to PSL_2(\mathbb{R}).$ 
 
We end this section by stating one of the most important theorems in the theory of arithmetic Fuchsian groups:

\

\textbf{Margulis' theorem:} For a finite volume  Fuchsian group $\Gamma$ the following three conditions are equivalent.
\begin{enumerate}
\item $\Gamma$ is arithmetic.
\item $\Gamma$ has infinite index in $\com(\Gamma).$
\item $\com(\Gamma)$ is a dense subgroup of $PSL_2(\mathbb{R})$ in the usual matrix topology.
\end{enumerate}

\textbf{Warning:} Given the invariant quaternion algebra $A\Gamma=\left(\dfrac{a,b}{k}\right)$ of the group $\Gamma$ uniformising an algebraic curve $C$  and an element $\sigma \in Gal( \C/\Q)$ one should 
 not confuse $A\Gamma^{\sigma}$
 with $(A\Gamma)^{\sigma}$. The first one is the invariant quaternion algebra of the Fuchsian group uniformising the algebraic curve $C^{\sigma}$ whereas the latter refers to the quaternion algebra  $\left(\dfrac{\sigma(a),\sigma(b)}{k^{\sigma}}\right).$ In general these two algebras are different, see Theorem \ref{mainI}.

\subsubsection{ An explicit example} \label{example}
We now  construct an arithmetic group derived from a quaternion algebra  over the real field $k=\mathbb{Q}(\sin \frac{2\pi}{p})$, where $p\geq 5$ is  a prime number. This example will be revisited later in the paper. We shall start with a simple observation relative to this field $k$.

\begin{lemma} \label{lemma1}
$k=\mathbb{Q}(\sin \frac{2\pi}{p})=\mathbb{Q}(\cos \frac{\pi}{2p}).$ In particular $k$ is a totally real field.
\end{lemma}
\begin{proof}
The identity  $\sin \frac{2\pi}{p}=-\cos\frac{(4+p)\pi}{2p}  $  shows that $k$ is the  subfield of the
\\[0.1cm]
cyclotomic field  $\mathbb{Q}(e^{(4+p)2\pi i/4p})=\mathbb{Q}(e^{2\pi i/4p})$ fixed by the complex conjugation, where the last equality holds because $4+p$ and $4p$ are co-prime. The result follows.
\end{proof}

With $k$ as above, our quaternion algebra is going to be

 $$A=\left(\dfrac{-1, \ b_p}{k}\right)
\text{ \  where \  \  } b_p=\cos (\frac{2\pi}{p})-1+\frac{32}{p^2}$$
 (instead of $1-\frac{32}{p^2}$ one can take  any other rational number lying between
 $\cos \frac{4\pi}{p}$  and $\cos \frac{2\pi}{p}$).

Let us check that $A$ satisfies the three  conditions required to apply the Borel--Harish-Chandra theorem. 

1. By Lemma \ref{lemma1} the field  $k$ is  totally real. 

2. Indeed  $A\otimes_k \mathbb{R}\cong M_2(\mathbb{R}),$ via the isomorphism that sends 
$\vec{1},\vec{i}, \vec{j}$  and $\vec{k} $ to

\

$  \left( \begin{array}{cc}
1 & 0  \\
0 & 1  \end{array}\right),   
 \left( \begin{array}{cc}
0 & 1 \\
-1 & 0 \end{array} \right),   
 \left( \begin{array}{cc}
\sqrt{b_p} & 0 \\
0 & -\sqrt{b_p}  \end{array} \right) \text{ \ and  \ }                            
 \left( \begin{array}{cc}
  0 & -\sqrt{b_p} \\
-\sqrt{b_p} & 0 \end{array} \right).$ 

 \

3. For any   $\tau:k \to \mathbb{ \mathbb{C}}$ different from the inclusion map we have 
$$b^{\tau}_p =\tau \left(\cos \frac{2\pi}{p}-1+\frac{32}{p^2}\right)=\cos \frac{2l\pi}{p}-1+\frac{32}{p^2}; \text{ for some } 2\leq l \leq (p-1)/2 $$
which is a negative real number, and therefore  
$A^{\tau}\otimes_k \mathbb{R}=\left(\dfrac{-1,b^{\tau}_p}{\mathbb{R}}\right) \cong \mathbf{H}.$

\

Now, the fact that the ring of integers of \ 
$k=\mathbb{Q}(\cos \frac{\pi}{2p})$ is  
$R_k= \mathbb{Z}[2\cos \frac{\pi}{2p}]$ (see e.g. \cite{Wa}, Proposition 2.16) allows us to write down an obvious explicit order of $A$, namely 
$$\begin{array}{ll}
\mathcal{O} & = R_k + R_k\vec{i}+R_k\vec{j}+R_k\vec{k}  \\ 
& \\
& \stackrel{\rho}\cong  \left\{ X= \left( \begin{array}{cc}
 a_1+a_3 \sqrt{b_p}& a_2-a_4 \sqrt{b_p} \\
& \\
-a_2-a_4 \sqrt{b_p} & a_1-a_3 \sqrt{b_p}  \end{array} \right): a_i \in \mathbb{Z}[2\cos \frac{\pi}{2p}] \right\}
\end{array}$$
According to what has been said above, for every prime number $p\geq 5$, the group
\begin{equation} \label{Explicitgroup}
\Gamma_p=P(\mathcal{O}^1)=\left\{  X \in \mathcal{O} : (a^2_1-a^2_3b_p)+(a^2_2-a^2_4b_p)=1 \right\}
\end{equation}
is going to  be a co-compact arithmetic Fuchsian group. Moreover, by Selberg's lemma  suitable finite index  subgroups of $\Gamma_p$ will be surface groups.


\subsection{ Solenoids} \label{solenoids}

Let  $(C,p)$ be a  pointed algebraic curve, or, equivalently, a pointed compact Riemann surface, of genus greater than one. We shall denote by   $\Lambda$  the collection  of all pointed unramified covers of 
$(C,p)$ so that $\lambda \in \Lambda$ stands for a triple $(C_{\lambda}, p_{\lambda}; f_{\lambda})$
where $f_{\lambda}:C_{\lambda}\to C$ is an unramified cover with $ f_{\lambda}(p_{\lambda})=p.$ Endowed with the partial order defined by $\lambda \leq \mu$ if there is an unramified cover
 $f_{\mu\lambda}:C_{\mu}\to C_{\lambda}$  
such that  $ f_{\mu, \lambda}(p_{\mu})=p_{\lambda}$ and 
$f_{\lambda}\circ f_{\mu\lambda}=f_{\mu},$ the set $\Lambda$ becomes a directed set.  The maps $f_{\mu\lambda}$ are sometimes called the \emph{bonding functions}. If we  denote by $o$
the element of $\Lambda$ representing the triple $(C, p; id.)$ then $ f_{\mu o}=f_{\mu}.$

 The family 
\begin{equation} \label{defsolenoid}
 \{ (C_{\mu},f_{\mu\lambda})\}_{ \lambda \leq \mu \in \Lambda}
 \end{equation}
forms an \emph{inverse system } of coverings of $C.$
By the \emph{solenoid associated  with}   $C,$ which we will denote by  
$
 \mathcal{H}_{C} = \mathop{\varprojlim}\limits _{\begin{subarray}{c}
\lambda
\end{subarray} }
 C_{\lambda},
$
we shall mean to the \emph{projective limit} of the inverse system (\ref{defsolenoid}). (Here and in the sequel we refer to \cite{RZ} for generalities on inverse systems and projective limits).
  
Suppose that we choose another base point $p'\in C$, then
a canonical identification between pointed covers of $(C,p)$ and pointed covers of $(C,p')$ can be obtained by making the triple $(C_{\lambda}, p_{\lambda}; f_{\lambda})$ correspond to the triple $(C_{\lambda}, p_{\lambda}'; f_{\lambda})$ where
$p_{\lambda}'$ is the endpoint   of the lift to $(C_{\lambda}, p_{\lambda};f_{\lambda})$ of a simple path in $C$ connecting $p$ to $p'$. Under this identification the inverse system  
 (\ref{defsolenoid})  remains unchanged.
This means that 
the definition of $\mathcal{H}_{C}$ is independent of the choice of the base point $p\in C.$
  However,  the fact that we work with pointed covers allows the following group theoretic interpretation  of $
 \mathcal{H}_{C}.$

Let $\mathbb{H}$ denote the upper half plane and let us fix
a pointed universal cover map $(\mathbb{H},*) \to (C,p).$
 Then we
  can  canonically  identify the fundamental group
$\pi_1(C,p)$ with the
  group    $\Gamma < PSL_2(\mathbb{R})$     of deck transformations
 of this cover
 and each fundamental group
$\pi_1(C_{\lambda},p_{\lambda})$
with the    subgroup    $\Gamma_{\lambda}<\Gamma$     of deck transformations of the covering
 $ (\mathbb{H}, *) \to (C_{\lambda},p_{\lambda})$. We then have commutative diagrams

 \begin{equation} \label{EquivalenceDiagram1}
\begin{array}{lclclcl}
\mathbb{H}/\Gamma_{\lambda} & \stackrel{\phi_{\lambda}}{\longrightarrow} & C_{\lambda} & & \mathbb{H}/\Gamma_{\mu} & \stackrel{\phi_{\mu}}{\longrightarrow} & C_{\mu} \\
\ \downarrow &  & \downarrow f_{ \lambda} &  \text{and} & \ \downarrow & & \downarrow f_{\mu \lambda}   \\
\mathbb{H}/\Gamma& \stackrel{\phi_o}{\longrightarrow} & C & &\mathbb{H}/\Gamma_{\lambda} & \stackrel{\phi_{\lambda}}{\longrightarrow} & C_{\lambda}
\end{array}
\end{equation}
where the vertical arrows on the left are the obvious projection maps induced by the corresponding inclusions of Fuchsian groups
and the horizontal ones are isomorphisms of Riemann surfaces uniquely determined by 
 the choice of $\phi_o$ and the condition $ \phi_{\lambda}([*])=p_{\lambda}.$ 
  All this allows us to alternatively rewrite  our solenoid as
 $ \mathcal{H}_{\Gamma}= \mathop{\varprojlim}\limits _{\begin{subarray}{c}
\lambda
\end{subarray} }
     \mathbb{H}/\Gamma_{\lambda},$
where $\Gamma_{\lambda}$ ranges among all finite index subgroups of $\Gamma$. The use of 
$\mathcal{H}_{C}$ or $ \mathcal{H}_{\Gamma}$ through the paper will tend to depend on whether we want to emphasize the algebraic or the hyperbolic nature of the solenoid.

We notice that for any    finite index subgroup of 
$K<\Gamma$  the family of  finite index subgroups of $K$ provides a co-final family 
of finite index subgroups of $\Gamma$ which means that 
$\mathcal{H}_{\Gamma} \equiv \mathcal{H}_{K}$. In other words, the solenoid 
$\mathcal{H}_{\Gamma}$
depends only on the commensurability class of $\Gamma.$
 
In the same vein we observe that in defining the solenoid we can restrict ourselves to \emph{the collection of finite index 	
  normal subgroups},   since they already form
   a co-final family, and so  we can   view 
	$\mathcal{H}_C$
   as the projective limit of pointed Galois covers of $C$, that is
   \begin{equation} \label{defsolenoidGalois}
	\mathcal{H}_{C} \equiv \mathcal{H}_{\Gamma}:= \mathop{\varprojlim}\limits _{\begin{subarray}{c}
 \Gamma_{\lambda}\triangleleft_{f} \Gamma
\end{subarray} }
     \mathbb{H}/\Gamma_{\lambda}
		\end{equation}
where the notation $\triangleleft_{f}$ stands for finite index normal subgroup.

We recall that  the\textit{ profinite completion} of $\Gamma$ is  precisely the group 
  \begin{equation} \label{GammaHatGalois}
\widehat{ \Gamma}:=\mathop{\varprojlim}\limits _{\begin{subarray}{c}
 \Gamma_{\lambda}\triangleleft_{f} \Gamma
\end{subarray} } \Gamma/\Gamma_{\lambda}
 \end{equation}
and that each of the finite groups $\Gamma/\Gamma_{\lambda}$ can be viewed as   the automorphism group of the covering $\mathbb{H}/\Gamma_{\lambda} \to \mathbb{H}/\Gamma$ which  through the identifications in  (\ref{EquivalenceDiagram1}) can be further identified  to $\aut(C_{\lambda}, f_{\lambda})$, the  group of automorphisms of the equivalent covering 
$f_{\lambda}:C_{\lambda} \to C.$
This means that the group $\widehat{ \Gamma}$ 
is isomorphic to the \emph{algebraic fundamental group} of $C$, usually denoted 
$\pi_1^{alg}(C,p):=  \mathop{\varprojlim}\limits _{\begin{subarray}{c}
\lambda
\end{subarray} }
      \aut(C_{\lambda}, f_{\lambda}).$ 
As the similarity of the expressions 
(\ref{defsolenoidGalois}) and 
 (\ref{GammaHatGalois}) suggests,
 there is a natural action of  the group $\widehat{ \Gamma}$ on the solenoid $\mathcal{H}_{C}$  which	will play an important role later on. 

It is sometimes convenient to 
choose a co-final sequence of finite index normal subgroups
  $$
  \cdots \Gamma_{n+1}<\Gamma_{n}<\Gamma_{n-1}< \cdots   <\Gamma_{1}=\Gamma; \text{ \ \ with  \ \ }   \bigcap_n \Gamma_n= \{ 1 \}.
	$$
and regard $\mathcal{H}_{\Gamma}$ (resp.
 $\widehat{ \Gamma}$) as the projective limit of   the Riemann surfaces 
$\mathbb{H}/ \Gamma_n$ (resp. the groups $\Gamma/ \Gamma_n$).
Typically one chooses $\Gamma_{n}$ be \emph{the standard characteristic sequence of  subgroups} of $\Gamma$ defined as
\begin{equation} \label{defGammaN}
\Gamma_{n}= \text{ intersection of  all subgroups of } \Gamma \text{ of index } \leq n.
\end{equation}
So from now on we will use freely the following alternative notation  
 \begin{equation} \label{LimitWithGammaN}
\begin{array}{c}
  \mathcal{H}_{\Gamma} =
\mathop{\varprojlim}\limits _{\begin{subarray}{c}
n
\end{subarray} }
    \mathbb{H}/\Gamma_{n} \equiv \mathop{\varprojlim}\limits _{\begin{subarray}{c}
n
\end{subarray} }  C_{n} =\mathcal{H}_{C}  \\
		\\
			\text{and} \\
			\\
	  \widehat{\Gamma} = \mathop{\varprojlim}\limits _{\begin{subarray}{c}
n
\end{subarray} }
     \Gamma/\Gamma_{n}
 \equiv \mathop{\varprojlim}\limits _{\begin{subarray}{c}
n
\end{subarray} }
      \aut(C_{n}, f_{n}) =\pi_1^{alg}(C,p),
			\end{array}
\end{equation}
where the algebraic curves $C_n$ and the bonding functions
$f_{n,n-1}:C_n\to C_{n-1}$ and $f_{n}:C_n\to C$ correspond 
to the Riemann surfaces $\mathbb{H}/\Gamma_{n}$ 
and the obvious projection maps $\pi_{n,n-1}:\mathbb{H}/\Gamma_{n} \to \mathbb{H}/\Gamma_{n-1}$ and $\pi_{n}:\mathbb{H}/\Gamma_{n} \to \mathbb{H}/\Gamma$ through the identifications in  (\ref{EquivalenceDiagram1}).

The group  $\widehat{ \Gamma}$ may also be seen as the standard Cauchy sequence completion of $\Gamma$ with respect to the  non-Archimedean   metric $d$ defined by the formula

	 $$
    d(\gamma_1, \gamma_2)=   
		   \dfrac{1}{n}       \mathrm{ \  \ \ \  if \  \ \  \  } \gamma_2^{-1}\gamma_1\in \Gamma_n \setminus \Gamma_{n+1}  
$$
	  Usually this completion is realized as  follows:
	\begin{equation} \label{GammaHat}
	\widehat{\Gamma}=\big\{\tau= (\tau_n)_n \in \prod_n (\Gamma/\Gamma_n): \tau_n \equiv \tau_{n-1} (mod \ \Gamma_{n-1}) \big\}
	\end{equation}
	The solenoid       $\mathcal{H}_{\Gamma}$
      can be canonically identified with the quotient of the product   
			$\mathbb{H} \times \widehat{\Gamma}$  under the action  of $\Gamma$ defined by 
			$$\gamma(z,\tau)=(\gamma(z), \tau \gamma^{-1})$$
	This	quotient	space we shall denote by 
  $\mathbb{H} \times_{\Gamma} \widehat{\Gamma}$ and the identification can be realized as follows 
	(see  \cite{Odd}, Theorem 3.5): 
   \begin{equation}\label{defsolenoid3}
  \begin{array}{rcl}
      \mathbb{H} \times_{\Gamma} \widehat{\Gamma} & \equiv &  \mathcal{H}_{\Gamma} \\
 \big[ z,\tau \big]  & \longleftrightarrow &     ( \tau_{n}(z) \in \mathbb{H}/\Gamma_{n}   )_{n}
  \end{array}
\end{equation}		
Both the hyperbolic metric 
$d_{h}$ on  
$\mathbb{H}$ and the non-archimedean metric 
$\hat {d}$
 on $\widehat{\Gamma}$ are $\Gamma$-invariant and so the corresponding product metric
on 
 $\mathbb{H} \times \widehat{\Gamma}$ descends to a well-defined metric 
$d_{\mathcal{H}}$ on $\mathcal{H}_{\Gamma}$
that induces the
inverse-limit topology. In explicit terms

$$
d_{\mathcal{H}}([ z_1,\tau_1],  [ z_2,\tau_2 ])= 
\inf_{g\in \Gamma}  \max\left\{ d_h(z_2,gz_1), \hat{d}(\tau_2, \tau_1 g^{-1})  \right\}
$$
 Let $r=1/n$ be smaller than the  injectivity radius of $ \mathbb{H}/\Gamma$ (or, equivalently, the minimum of the translation length of the elements in 
$\Gamma\setminus \{\Id\}$); then  the  ball of radius $r$ in $\mathcal{H}_{\Gamma}$ centered at a point 
$[z,\tau]$ equals  $D\times T$  where $D$ is the  hyperbolic disc of radius $r$ centered at $z$  and $T$ is the coset $\tau\widehat{\Gamma}_{n+1}.$ We see that 
$\mathcal{H}_{\Gamma}$ carries a complex structure with respect to the $z$-variable. Specifically, a continuous function $f(z,\tau)$ defined on an open set of  $\mathcal{H}_{\Gamma}$ will be \emph{said  to be holomorphic} if it is so with respect to the $z$-variable.
 As a topological space $ \mathcal{H}_{\Gamma} $ is compact and connected, but not path connected. There are natural projections
 $$
\begin{tikzpicture}[node distance=1.7 cm, auto]
  \node (P) {$ \   \  \  \  \    \mathcal{H}_{C}\equiv \mathbb{H} \times_{\Gamma} \widehat{\Gamma}$};
  \node (A) [below of=P, left of=P] {$ \  \ \   \  \   \   \  \  C \equiv \mathbb{H} /\Gamma$};
  \node (C) [below of=B, right of=P] {$\widehat{\Gamma}/\Gamma$};
  \draw[->] (P) to node [swap] {${\pi}_1$} (A);
  \draw[->] (P) to node {${\pi}_2$} (C);
\end{tikzpicture}
$$
where $\widehat{\Gamma}/\Gamma$ stands for the set of left co-sets.
The fibers of $\pi_1$ are totally disconnected compact sets identifiable to
  $\widehat{\Gamma}$ and   the fibres of $\pi_2$,
  called \emph{the leaves}, are the path connected
   components,   each of them being a dense   subset of
  $ \mathcal{H}_{C} .$  They are parametrized by $\widehat{\Gamma}/\Gamma$. The leaf

   $$
   \mathcal{B}_{C}= \mathcal{B}_{\Gamma}:=\pi_2^{-1}([1])= \mathbb{H} \times_{\Gamma} \{ 1 \}   $$
   will be called \emph{the base leaf} and it is conformally equivalent to $ \mathbb{H}.$
	
	Riemann surface solenoids were introduced by D. Sullivan in \cite{Sul} and were then studied by several authors, including Sullivan himself, Biswas, Markovic, Nag, Odden,  Penner,  \v{S}ari\'c
	and several others (\cite{BN}, \cite{MS}, \cite{Odd}, \cite{PS}).

	\subsubsection{ Automorphisms of a solenoid} \label{Autsolenoids}
	We start with a couple of observations 
\begin{lemma} \label{lemma} 
Let $\{\Gamma_{n}\}_n$ be the standard sequence of characteristic subgroups introduced in (\ref{defGammaN}). Then
\begin{enumerate}
 \item An element  $\alpha \in PSL_2( \mathbb{R})$ lies in 
$\com(\Gamma)$
  if and only if for any index $n$
	there is a finite index subgroup $K_n<\Gamma_n$ such that $\alpha K_n \alpha^{-1}$ is also a finite index subgroup of $\Gamma_n.$
\item Let $\alpha \in \com(\Gamma)$. There is an increasing sequence of positive integers $a_n=a_n(\alpha)$
such that   $\alpha \Gamma_{a_n}\alpha^{-1}<\Gamma_{n}.$  
\end{enumerate}
\end{lemma}
\begin{proof}
(1) Since $\com(\Gamma_n)=\com(\Gamma)$, it is enough to prove it for $\Gamma_1=\Gamma.$
In one direction this is proved  by letting $K$ be $\Gamma\cap \alpha^{-1}\Gamma \alpha.$  For
the converse  statement note that now by hypothesis 
$\Gamma\cap \alpha\Gamma \alpha^{-1}$ contains 
$ \alpha K \alpha^{-1}$ 
which is a finite index subgroup of both 
$\Gamma$ and $ \alpha\Gamma \alpha^{-1}$, hence
$\Gamma\cap \alpha\Gamma \alpha^{-1}$ must also be 
is a finite index subgroup of both 
$\Gamma$ and $ \alpha\Gamma \alpha^{-1}.$
 
 (2)   follows by choosing $a_n$  
such that $\Gamma_{a_n}$   is contained in $K_{n}.$ 
\end{proof}

The first part of Lemma \ref{lemma} allows us to regard
the elements   $\alpha \in \com(\Gamma)$ as
  automorphisms (i.e. holomorphic bijections) of
  $ \mathcal{H}_{\Gamma}$ in the following way.  
	   Choose a finite index subgroup $K< \Gamma$ such that 
	$\alpha K\alpha^{-1}< \Gamma$. Then the automorphism induced by $\alpha$ results as composition of the following three isomorphisms: Start with the identification $\mathcal{H}_{\Gamma}  \equiv  \mathcal{H}_{K}$, then compose with the obvious isomorphism 
	$\mathcal{H}_{K} \cong \mathcal{H}_{\alpha K \alpha^{-1}}$ and then with the identification  $\mathcal{H}_{\alpha K \alpha^{-1}} \equiv \mathcal{H}_{\Gamma}.$ In more explicit terms: 
	\begin{equation} \label{DefinitionAlpha}
	\begin{array}{lcclclc}
	\mathcal{H}_{\Gamma} & \equiv & \mathcal{H}_{K} & \xrightarrow{\sim} &  
	\mathcal{H}_{\alpha K \alpha^{-1}} & \equiv & \mathcal{H}_{\Gamma}   \\    \left[z,\tau\right] & \equiv & [\gamma z,\tau' ]  &  \longrightarrow & [ \alpha \gamma z, \alpha \tau' \alpha^{-1} ] & \equiv &  [ \alpha \gamma z, \alpha \tau' \alpha^{-1} ] 
\end{array}
\end{equation}
  where $\tau' \in \widehat{K}, \gamma\in \Gamma$ and 
	$\tau=\tau'\gamma.$ (Here we are using the identification  $\widehat{K}/K \cong \widehat{\Gamma}/\Gamma$, which is a direct consequence  of the elementary identity 
	$\Gamma/K \cong \widehat{\Gamma}/\widehat{K}$).
	
	Note that on the base leaf $\alpha$ acts simply  by the formula 
	$\alpha[z,\Id]=[\alpha z,\Id]$, thus $\alpha$ is base leaf preserving. Conversely, 
	   let $F:\mathcal{H}_{\Gamma} \to \mathcal{H}_{\Gamma}$ be an automorphism preserving the base leaf  $\mathcal{B}_{\Gamma} \equiv  \mathbb{H} $
then, clearly,  $ F([ z,\Id] =[\alpha z, \Id]$ for some  $\alpha=\alpha_F    \in PSL_2(\mathbb{R})$ and it turns out that, in fact,  $\alpha \in \com(\Gamma).$ This result can be found in \cite{Odd}, Corollary 4.8 and   \cite{BN}, Proposition 4.12).

On the other hand an element  $\tau\in\widehat{\Gamma}$
  acts as an
	automorphism of
  $\mathcal{H}_{\Gamma} $  simply by the rule
  \begin{equation} \label{actionGammahat}
  \tau([z,\tau'])=[z, \tau \tau']
  \end{equation} 
This formula defines a transitive action of $\widehat{\Gamma}$ on the set of leaves
  which agrees with the standard  group action of 
$\widehat{\Gamma}$ on $\widehat{\Gamma}/\Gamma$.
  This action  is base leaf preserving only when $\tau\in \Gamma.$ In other words, one has
\begin{equation} \label{Inters}
 \widehat{\Gamma} \cap \com(
 \Gamma) =  \Gamma
\end{equation}

 Now, let $F\in \aut (\mathcal{H}_{\Gamma})$ be an arbitrary automorphism of  $\mathcal{H}_{\Gamma}$ and suppose that $F$   sends the base leaf $\pi_2^{-1}([1])$ to another leaf 
$\pi_2^{-1}([\tau])$, then $\tau^{-1}\circ F$   will be a base leaf preserving automorphism, thus we see that
\begin{equation} \label{aut(H)}
 \aut(\mathcal{H}_{\Gamma})=\widehat{\Gamma} \cdot 
\com(\Gamma) =   \com( \Gamma)\cdot \widehat{\Gamma}
\end{equation}
where the second equality holds because the product of the two subgroups is a group.


\section{The  group structure of $\aut (\mathcal{H}_{\Gamma})$ }
 
Turning to the identity (\ref{aut(H)}) above we should mention 
that, in fact, Odden 
has shown (see Theorem 4.13 and Corollary 4.14 in  \cite{Odd}) that  the rule $(\tau,\alpha) \to \tau \alpha$ yields a bijection between the set $ \widehat{\Gamma}\times_{\Gamma} \com(\Gamma)$ (where 
	$\Gamma$ acts by the rule $\gamma (\tau,\alpha)=(\tau \gamma^{-1}, \gamma \alpha)$)
	and $\aut (\mathcal{H}_{\Gamma}).$ The author warns however that \textit{``(this) theorem 
does not shed light on its group structure"}. This group structure will  turn out to be a certain completion of $\com(\Gamma)$ whose description is the goal of this section.

	\subsection {The Belyaev completion of $\com(\Gamma)$.} \label{Belyaev}
	
 A  Hecke pair is a pair of groups $(G,H)$ where $H$ is a
subgroup of G such that  $H$ and $gHg^{-1}$ are commensurable for all $g \in G.$ If $H$ is residually finite  the natural left action of $G$ on the sets of cosets  
 $$\mathcal{B}:=\{xK: x\in G  \text{ and } K \text{ commensurable with } H \}$$
determines a faithful permutation represention 
                $G \hookrightarrow \sym( \mathcal{B}).$

Endowed with the the topology of pointwise convergence arising from the discrete topology on  
 $\mathcal{B}$ the group $\sym( \mathcal{B})$ of permutations of $\mathcal{B}$ becomes a Hausdorff topological group (see \cite{Be}, see also  
\cite{KLQ} and \cite{RW}).
As $xK$ varies in $\mathcal{B}$ the collection of stabilisers
		$$\sym( \mathcal{B})_{xK}=\{\phi \in  \sym( \mathcal{B}): \phi(xK)=xK\}$$
	provides a  subbase  of neighbourhoods of the identity   for this topology. Accordingly, the collection
		of   stabilisers
		$$G_{xK}=\sym( \mathcal{B})_{xK}\cap G=xKx^{-1},
		\text{ \ with \ } xK \in  \mathcal{B} $$
		is a   subbase  of neighbourhoods of the identity  for the induced topology on $G.$

Here we will be concerned with the Hecke pair 	$(\com(\Gamma), \Gamma)$ where, as before,  $\Gamma$  stands for the Fuchsian group uniformising a compact Riemann surface of genus $g\geq 2.$
We will denote by $\overline{\com(\Gamma)}$ the closure of (the image of) $\com(\Gamma)$
in $\sym( \mathcal{B})$ and we will refer to this group  
as the 
\emph{Belyaev completion} of $\com(\Gamma)$.

  A feature of this completion is that it contains the group $\widehat{\Gamma}$ as the closure of   $ \Gamma$.  This can be seen by considering   the following chain of topological groups
$$
\Gamma< \prod_n (\Gamma/\Gamma_n) <  \prod_n \sym(\Gamma/\Gamma_n) <\sym(\mathcal{B}^*) <  \sym(\mathcal{B})
$$
where 
$
\mathcal{B}^*=\{x \Gamma_n: x\in \com(\Gamma), n\in \mathbb{N} \} \subset \mathcal{B}
$.  Since $ \prod_n (\Gamma/\Gamma_n)$ is compact, hence closed,  the closure of $\Gamma$  in 
$\sym(\mathcal{B})$ agrees with its closure in $\prod_n (\Gamma/\Gamma_n)$, but this is one of the definitions of $\widehat{\Gamma}$; see (\ref{GammaHat}).

Similarly, since  $\sym(\mathcal{B}^*)$ is clearly a closed subset of $\sym(\mathcal{B})$, the closures of $\com(\Gamma)$ in $\sym(\mathcal{B}^*)$ and     $\sym(\mathcal{B})$ agree.

The next proposition collects the main properties of  $\overline{\com(\Gamma)}.$
\begin{proposition} \label{ComHat}
\begin{enumerate}
\item $\overline{\com(\Gamma)}$ is a locally compact subgroup of 
$\sym( \mathcal{B}^*)<\sym( \mathcal{B})$
such that the closure 
 $\overline{H}$
of each subgroup
 $H$ commensurable with $\Gamma$ 
is a  compact open subgroup.
	\item  $\overline{\Gamma}$ can be identified to $\widehat{\Gamma}$. 
	\item  $\overline{\com(\Gamma)}=\widehat{\Gamma} \cdot 
\com(\Gamma) =   \com( \Gamma)\cdot \widehat{\Gamma}$
	\item  $\overline{\com(\Gamma)}/ \ \widehat{\Gamma} \equiv
\com(\Gamma)/\Gamma $   
\end{enumerate}
		\end{proposition}
		\begin{proof}
		(1) is the content of Theorem	7.1 in \cite{Be}.	
		
(2) is the comment preceding this proposition.

(3) Let $x\in \overline{\com(\Gamma)}.$ Applying part (1) 
to $H=\Gamma$ we see that 
$x \overline{\Gamma} \cap \com(\Gamma)$ is non-empty, so there is some $\alpha \in \com(\Gamma)$ such that $\alpha=x \tau$ for some 
$\tau\in \widehat{\Gamma}$, hence $x\in  \com( \Gamma)\cdot \widehat{\Gamma}=
\widehat{\Gamma} \cdot \com( \Gamma).$

(4) follows from (3).
\end{proof}

\begin{remark} (\cite{KLQ}, Example 2.4).
We observe that although $\overline{\com(\Gamma)}$ is complete (being locally compact), the group $\sym(\mathcal{B})$ itself is not. For instance the sequence of functions 
$$
\phi_n(xK) =
\left\{
	\begin{array}{ll}
		xK  & \mbox{if } xK \neq \Gamma_j, \text{ with } 1\leq j \leq n \\
		\Gamma_{i+1} & \mbox{if } xK = \Gamma_i, \text{ with } i < n \\
		\Gamma_{1} & \mbox{if }   xK = \Gamma_n
	\end{array}
\right.
$$
 converges to the shift map $\Gamma_i \to \Gamma_{i+1}$ which is not a bijection.   In fact in \cite{Be} the completion of $G$ is defined to be the closure $\overline{G}$ of $G$ in $Map(G)$, the semigroup of maps of $G$, and it is then shown that $\overline{G}<\sym(\mathcal{B}).$
 \end{remark}
	
	\subsection {The isomorphism  $\aut(\mathcal{H}_{C}) \cong \overline{\com(\Gamma)}$ } 
	
\begin{lemma} \label{isomorphism of quotients}
  The inclusion $\com(\Gamma)< \aut(\mathcal{H}_{\Gamma})$ induces the following natural bijection between sets of cosets
$$
\begin{array}{rll} \com(\Gamma) / \Gamma_n & \simeq & \aut(\mathcal{H}_{\Gamma}) / 
 \  \widehat{\Gamma}_n  \\
x\Gamma_n & \to & x \widehat{\Gamma_n}
\end{array}
        $$ 
	\end{lemma}
	\begin{proof}
When $n=1$,   $\Gamma_n=\Gamma,$ and  the result follows  from (\ref{Inters}) and (\ref{aut(H)}). But the result holds for any $n$ because 
$ \com(\Gamma_n)=\com(\Gamma)$ and $\mathcal{H}_{\Gamma_n}=\mathcal{H}_{\Gamma}.$
  \end{proof}

The bijections  $ \com(\Gamma)/  \Gamma_n \simeq  \aut(\mathcal{H}_{\Gamma}) /
  \widehat{\Gamma}_n $  permit us to identify the sets
	$\mathcal{B}^*=\{x \Gamma_n: x\in \com(\Gamma), n\in \mathbb{N} \}$ and $\widehat{\mathcal{B}}=\{F \widehat{\Gamma_n}: F\in \aut(\mathcal{H}_{\Gamma}), n\in \mathbb{N} \} $ and therefore to
	transfer the natural action of  $\aut(\mathcal{H}_{\Gamma})$ on $\widehat{\mathcal{B}}$
		to an action  on
	$\mathcal{B}^*$.

	\begin{lemma} \label{TransferingActions}
	\begin{enumerate}
	\item The   action of $\aut(\mathcal{H}_{\Gamma})$ on  
	$\mathcal{B}^*$ introduced above 
	is given by the formula
			  $$
	\begin{array}{rclll}
	\aut(\mathcal{H}_{\Gamma})&\times& \mathcal{B}^* &  \longrightarrow & \mathcal{B}^* \\
	\ \  \ \ (F&, & x \Gamma_n) & \longrightarrow & \alpha x \Gamma_n
\end{array}
$$
where $ \alpha= \alpha(F,x,n)\in \com(\Gamma)$   is such that
 $\alpha^{-1}F \in x\widehat{\Gamma}_n x^{-1}.$	
	\item This action is faithful  and so it yields a permutation representation of 
	$\aut(\mathcal{H}_{\Gamma})$ in $\sym(\mathcal{B}^*)<\sym(\mathcal{B}).$
		\item Its restriction to $\com(\Gamma)$ is the standard action $(\alpha,  x \Gamma_n)  \to \alpha x \Gamma_n.$
	\end{enumerate}
	\end{lemma}
	\begin{proof}
	1) The condition imposed on $\alpha$ only means   that 
	$F  x \widehat{\Gamma}_n=\alpha  x \widehat{\Gamma}_n.$
	
	2) and 3) are obvious.
	\end{proof}

	This lemma together with Proposition \ref{ComHat}
	allows us to regard both groups $\overline{\com(\Gamma)}$  and  $\aut(\mathcal{H}_{\Gamma})$ as subgroups of $\sym( \mathcal{B}^*).$

	\begin{theorem} \label{Isom}
	$$\overline{\com(\Gamma)}= \aut(\mathcal{H}_{\Gamma})$$  
 \end{theorem}
\begin{proof}
Let us denote by $ j:\aut(\mathcal{H}_{\Gamma})  \hookrightarrow \sym(\mathcal{B}^*)$ the permutation representation described in Lemma \ref{TransferingActions}. What  we have to prove is  that  
$\overline{\com(\Gamma)}= j(\aut(\mathcal{H}_{\Gamma})).$
By the third part of this lemma we only need to show that  1)	$j(\aut(\mathcal{H}_{\Gamma}))$ is a closed subset of $\sym( \mathcal{B}^*)$ and 2)  $j(\com(\Gamma))$ is dense in it. 

1) $j(\aut(\mathcal{H}_{\Gamma}))$ is a closed subset of $\sym( \mathcal{B}^*).$

We start with the observation that  $\com(\Gamma)$ is countable. Indeed in the non-arithmetic case $\com(\Gamma)$ is still discrete and  in the arithmetic case even the whole invariant quaternion algebra $A\Gamma$ is countable. So let 
$$
\com(\Gamma)=\{x_1=1, x_2, \cdots , x_n, \cdots\}
$$
be an enumeration of its elements.
Then we can inductively construct a co-final subsequence 
$\Gamma_{a_1}>\Gamma_{a_2}>\cdots  >\Gamma_{a_n}> \cdots  $ of our standard sequence of characteristic subgroups 
$\{\Gamma_n\}$ enjoying the property  
$$
\Gamma_{a_n}< \bigcap_{i,j\leq n} x_i\Gamma_jx_i^{-1}
$$

Let now $f\in \sym( \mathcal{B}^*)$ be an element in the closure of $j(\aut(\mathcal{H}_{\Gamma}))$. We want to show that $f$ lies in fact in  $j(\aut(\mathcal{H}_{\Gamma}))$. In order to do that, 
 for any 
$n \in \mathbb{N}$, we consider the following neighbourhoods of $f$ in $\sym( \mathcal{B^*}):$
$$V(f,n)=\{h\in \sym( \mathcal{B^*}): h(x_i\Gamma_{a_j})=f(x_i\Gamma_{a_j}); i,j\leq n \}=f \cdot \bigcap_{i,j\leq n} \sym( \mathcal{B}^*)_{\Gamma_{a_i}}
$$
Let $F_n\in \aut(\mathcal{H}_{\Gamma})$ such that 
 $j(F_n)\in V(f,n) \cap j(\aut(\mathcal{H}_{C})).$  By Lemma \ref{TransferingActions}, 
$ j(F_n)(\Gamma_{a_n})=\alpha_n\Gamma_{a_n}$ with $\alpha_n \in \com(\Gamma)$ such that 
 $\alpha_n^{-1}F_n \in \widehat{\Gamma}_{a_n}<\widehat{\Gamma}_{a_{n-1}}.$ Therefore 
$ j(F_n)(\Gamma_{a_{n-1}})=\alpha_n\Gamma_{a_{n-1}}.$ But, on the other hand,
 $j(F_n)(\Gamma_{a_{n-1}})=f(\Gamma_{a_{n-1}})=j(F_{n-1})(\Gamma_{a_{n-1}})=\alpha_{n-1}\Gamma_{a_{n-1}}$. The conclusion is that  $\alpha_{n-1}^{-1} \alpha_{n} \in \Gamma_{a_{n-1}}.$
 Thus, we can write
$$
\begin{array}{ll}
\alpha_2=\alpha_1 \gamma_1, & \gamma_1\in \Gamma_{a_1} \\
\alpha_3=\alpha_2 \gamma_2= \alpha_1 \gamma_1 \gamma_2, & \gamma_2\in \Gamma_{a_2}  \\ 
\cdots \cdots   \cdots  \cdots\cdots\cdots\cdots&  \cdots \\
\alpha_i=\alpha_1 \gamma_1 \gamma_2\cdots \gamma_{i-1}, & \gamma_{i-1}\in \Gamma_{a_{i-1}}
\end{array}
$$
 Now, put $\tau_i=\gamma_1 \gamma_2\cdots \gamma_{i} \ (mod \ \Gamma_{a_{i+1}}).$ 
Then, clearly, the sequence $\tau=(\tau_i)_i$ defines an element  of \ 
$\widehat{ \Gamma}=\mathop{\varprojlim}
  \Gamma/\Gamma_{a_i}.$ We claim that  $f=j(F)$, where 
$F=\alpha_1 \tau \in \aut(\mathcal{H}_{C}).$

 Indeed, $f$ agrees with $j(F)$ at the points 
$\Gamma_{a_{n}} \in  \mathcal{B^*}$  because on the one hand
$f(\Gamma_{a_{n}})= j(F_n)(\Gamma_{a_{n}})=\alpha_n\Gamma_{a_{n}}$ and on the other hand
$$
\alpha_n^{-1}F=\alpha_n^{-1}\alpha_1 \tau =
(\tau_{n-1}^{-1}\tau_i)_i\equiv \left(\gamma_n \gamma_{n+1}\cdots \gamma_{n+k-1} \ (mod \ \Gamma_{a_{n+k}})\right)_k\in \widehat{\Gamma}_{a_{n}}.
$$
Moreover, since $ \widehat{\Gamma}_{a_{n}} <x_i\widehat{\Gamma}_j x_i^{-1}$ for  $n\geq i,j$, the above relation also     proves that  
$F(x_i\Gamma_j )=\alpha_n x_i\Gamma_j=F_n(x_i\Gamma_j)=f(x_i\Gamma_j)$ and we conclude that $j(F)$ agrees with $f$ at all points of $\mathcal{B^*}$, as claimed.

2)  $j(\com(\Gamma))$ is dense in $j(\aut(\mathcal{H}_{\Gamma}))$.

Our subbasis of neighbourhoods of the identity 
in $j(\aut(\mathcal{H}_{\Gamma}))$ 
consists of the sets 
$j(T\widehat{\Gamma}_n T^{-1}), T \in \aut(\mathcal{H}_{\Gamma})$;  but since, by (\ref{aut(H)}),
each $T$ equals to a product of 
 the form 
 $T=x \tau'$, with $x \in \com(\Gamma)$ and 
$\tau' \in \widehat{\Gamma}$, this subbasis can be rewritten as   
$\{j(x\widehat{\Gamma}_n x^{-1}): x\in  \com(\Gamma)  \}_n.$ This
  makes it clear that any set in this subbasis contains one of the form 
$\widehat{\Gamma}_m $ for sufficienty large $m.$
  It follows that if $F\in \aut(\mathcal{H}_{\Gamma})$ any
neighbourhood  of $j(F)$  contains a smaller one of the form 
$V_m=j(F \cdot \widehat{\Gamma_m}).$ Thus, what we need to see is that for each $m$ the intersection $V_m\cap j( \com(\Gamma))$ is nonempty. 
Now, setting  $F=\alpha \tau$, with 
$\alpha \in \com(\Gamma)$ and 
$\tau=(\tau_i)_i \in \widehat{\Gamma}$ we see that, as before, 
$(\alpha \tau_m)^{-1}F= \tau_m^{-1} \tau \in \widehat{\Gamma_{m+1}}<\widehat{\Gamma_{m}}.$
 This means that 
	$j(\alpha \tau_m ) \in V_m\cap j( \com(\Gamma))$ as required.
 \end{proof}

\section{Galois action on solenoids and commensurators}
\subsection{The group $\overline{\com(\Gamma)}$ is Galois invariant}
 Let $Gal( \C/\Q)$  denote the group of field automorphisms of $\C$, and let
 $\sigma \in Gal( \C/\Q)$.  The obvious action of $Gal( \C/\Q)$ on complex algebraic curves 
 transforms pointed unramified covers of $(C,p)$ into pointed unramified covers of  $(C^{\sigma},p^{\sigma})$  thereby yielding a natural bijection
  \begin{center}
 $\begin{array}{llll}
 \sigma:& \mathcal{H}_{C}  & \longrightarrow & \mathcal{H}_{C^{\sigma}}  \\
 & (q_n)  & \longrightarrow & (q_n^{\sigma}) 
\end{array}$
\end{center}
where   $\mathcal{H}_{C}$ is regarded as the limit of the inverse system    $\big\{  \big(C_{n}, f_{n, n-1} \big) \big\}_n$	introduced in   (\ref{LimitWithGammaN}) so that 
	$q_n$ stands for a point in $C_n$ such that
		$f_{n, n-1}(q_n)=q_{n-1}.$

Our first goal is to show that although this bijection is far from being an isomorphism (even far from being continuous) it induces an isomorphism between the corresponding automorphism groups 
$\aut(\mathcal{H}_{C})$ and 
$\aut(\mathcal{H}_{C^{\sigma}}).$ In order to do that we  first need to write  the elements of $\aut(\mathcal{H}_{C})$ in a convenient way. 

Notice that to any increasing subsequence of positive integers 	
			$\{a_n\}_n$  there corresponds a co-final subsequence of standard characteristic  subgroups $\{\Gamma_{a_n}\}_n$  and algebraic curves 
	$C_{a_n}$	which still define the solenoid, that is  	
		$
\mathcal{H}_{C} = \mathop{\varprojlim}\limits 
     C_{a_n}.$
	 Let 
 \begin{equation} \label{IsomInvSystems}
  F=\{F_{n}\}_{n} :
 \Big\{  \Big(C_{a_n}, f_{a_n, a_{n-1}} \Big) \Big\}_{n} \longrightarrow
 \Big\{ \Big(C_{n }, f_{n,n-1} \Big) \Big\}_{n }
 \end{equation}
be a \textit{morphism of inverse systems}.
  By that 
we mean (see  \cite{RZ}) that each
$F_{n}$ is a rational (or, equivalently, holomorphic)  map between $C_{a_n}$ and $C_{n}$
 such that  for any  $n \geq 2$
the following diagram commutes
 \begin{center}
 $\begin{array}{lll}
 C_{a_n}  & \stackrel{F_{n }}{\xrightarrow{\hspace*{1cm}}} & C_{n}  \\
 \Big\downarrow{ f_{a_na_{n-1}}} &  &   \Big\downarrow{f_{n,n-1}}  \\
 C_{a_{n-1}}   & \stackrel{F_{n-1 }}{\xrightarrow{\hspace*{1cm}}} & C_{n-1}  
\end{array}$
\end{center}

   Clearly,  in this situation, $F$ defines a holomorphic map  $F:\mathcal{H}_{C} \to \mathcal{H}_{C}.$
Our next result states that any automorphism of  $\mathcal{H}_{C}$ can be written in this way.

\begin{proposition} \label{Allarestandard}
Every automorphism $F$ of $\mathcal{H}_{C}$ is induced by a morphism of  inverse limits as in (\ref{IsomInvSystems}).
\end{proposition}
 \begin{proof}
We deal with the two different kinds of automorphisms separately.

1) $F=\tau=(\tau_{n})_{n} \in
	\widehat{\Gamma}   =\mathop{\varprojlim}\limits   \Gamma/\Gamma_{n}.$

	From the expression  of the automorphism   $\tau$ on the model of the solenoid 
	$\mathbb{H} \times_{\Gamma} \widehat{\Gamma}$ given in  (\ref{actionGammahat}) and the identification between the two models of the solenoid  made in (\ref{defsolenoid3}) one infers that in this case $F$ is the automorphism induced by the morphism of inverse limits
	$
	\{\tau_{n}\}_n:\{(\mathbb{H}/\Gamma_{n}, \pi_n)\}_n \to \{(\mathbb{H}/\Gamma_{n}, \pi_n)\}_n
	$
	or its corresponding algebraic version $ \{F_{n}\}_{n}: \{C_n, f_n\}_n\to
	\{C_n, f_n\}_n .$

2) $F= \alpha \in \com(\Gamma).$

Let  $\{\Gamma_{n_{\alpha}}\}$
 the co-final family of subgroups of $\Gamma$
 whose existence is  guaranteed  in part 2) of Lemma \ref{lemma} and consider the   holomorphic surjections 
$A_n:\mathbb{H}/\Gamma_{n_{\alpha}} \to \mathbb{H}/\Gamma_{n}$
defined by the following composition of maps
 \begin{center}
 $\begin{array}{ccccc}
\mathbb{H}/\Gamma_{n_{\alpha}} & \stackrel{\alpha }{\longrightarrow} & \mathbb{H}/\alpha\Gamma_{n_{\alpha}}\alpha^{-1} & \longrightarrow & \mathbb{H}/\Gamma_{n}  \\
z & \longrightarrow & \alpha(z) & \longrightarrow & \alpha(z)
\end{array}$
\end{center}
Clearly, these maps define a morphism of inverse systems  
$\{A_n\}:\{(\mathbb{H}/\Gamma_{n_{\alpha}},\pi_{n_{\alpha}}) \} \to 
\{(\mathbb{H}/\Gamma_{n}, \pi_{n})\}$  which yields a holomorphic map  $A:\mathcal{H}_{\Gamma} \to\mathcal{H}_{\Gamma} $ which coincides with $\alpha$ on the base leaf and, by continuity, on the whole solenoid. 
Now we only need to replace the holomorphic maps $A_n:\mathbb{H}/\Gamma_{\alpha_n} \to \mathbb{H}/\Gamma_{n}$ and $\pi_{n,n-1}:\mathbb{H}/\Gamma_{n} \to \mathbb{H}/\Gamma_{n-1}$ by their corresponding rational maps $F_n:C_{\alpha_n} \to C_{n}$ and  $f_{n,n+1}:C_{n} \to C_{n-1}.$

3) Finally,  for an arbitrary $F\in \aut(\mathcal{H}_{C})$
we can use (\ref{aut(H)}) 
to write $F$ in the form $F=\alpha \circ \tau$ and then the conclusion follows by composition of the two previous cases. 
	 \end{proof}

 \begin{proposition} \label{autHSigma} 
(1) For any $\sigma \in Gal( {\C}/\Q)$  the bijection $ \sigma: \mathcal{H}_{C} \to \mathcal{H}_{C^{\sigma}} $ introduced at the beginning of this section induces
  a group isomorphism
$$
	\begin{array}{lcl}
	\aut(\mathcal{H}_{C}) &  \longrightarrow & \aut( \mathcal{H}_{C^{\sigma}} ) \\
	\ \  \ \ F & \longrightarrow & F^{\sigma}:=\sigma \circ F \circ \sigma^{-1}
\end{array}
$$
(2) This isomorphism maps $\pi_1^{alg}(C,p)  \equiv \widehat{ \Gamma}$ onto  $\pi_1^{alg}(C^{\sigma},p^{\sigma})  \equiv\widehat{\Gamma^{\sigma}}.$ \newline
(3) If the pair $(C,p)$  is defined over a number field
one can replace in the above statements the group 
$Gal( {\C}/\Q)$ by  the absolute Galois group
   $Gal( \overline{\Q}/\Q).$    
 \end{proposition}
     \begin{proof}
  (1) By Proposition \ref{Allarestandard} we can assume that  $F \in \aut(\mathcal{H}_{C})$ is given by a morphism of inductive limits of the form
 $$
  F=\{F_{n}\}_{n} :
 \Big\{  \Big(C_{a_n}, f_{a_n} \Big) \Big\}_{n} \to
 \Big\{ \Big(C_{n }, f_{n } \Big) \Big\}_{n}
 $$
	and then  $F^{\sigma}$ will be the automorphism of 
	$\mathcal{H}_{C^{\sigma}}$ induced by  
	 $$
  F^{\sigma}=\{F^{\sigma}_{n}\}_{n} :
 \Big\{  \Big(C_{a_n}^{\sigma}, f_{ a_n}^{\sigma} \Big) \Big\}_{n} \to
 \Big\{ \Big(C^{\sigma}_{n}, f^{\sigma}_{n } \Big) \Big\}_{n}$$
where $ F^{\sigma}_{n}=\sigma \circ F_{n} \circ \sigma^{-1}$. We point out that 
each map
$ F^{\sigma}_{n}:C_{a_n}^{\sigma}  \to  C_{n}^{\sigma}$
is a rational (i.e. holomorphic) map,  in fact the rational map
whose defining equations are obtained by applying $\sigma$ to the defining equations of the rational map  $F_{n}:C_{a_n}  \to C_{n}.$ 
We also observe that since $\{ (C_{a_n}, p_{a_n} ) \}_{n} $ is a co-final family of pointed coverings of $(C, p)$ the family 
 $\{ (C_{a_n}^{\sigma}, p_{a_n}^{\sigma} ) \}_n$ is a co-final family of coverings of $(C^{\sigma},p^{\sigma})$. We thus conclude that 
	 $\mathcal{H}_{C^{\sigma}}$ is the projective limit of the inverse system  
$\{ (C_{a_n}^{\sigma}, f_{ a_n}^{\sigma} ) \}_{n}	$ and that $F^{\sigma}$ induces an automorphism of  $\mathcal{H}_{C^{\sigma}}$.

 (2) Clearly, if in the discussion above $F \in \pi_1^{alg}(C,p)=  \mathop{\varprojlim}\limits 
      \aut(C_{n}, f_{n})$, 	then 
	 $F^{\sigma} \in \pi_1^{alg}(C^{\sigma},p^{\sigma})=  \mathop{\varprojlim}\limits 
      \aut(C_{n}^{\sigma}, f_{n}^{\sigma})$, 
 as stated.
	
	(3) This statement follows from the observation that if $(C,p)$ is defined over $\overline{\Q}$ then,  the covering pointed curves $(C_{n},p_{n})$ as well as the automorphisms $F_{n}$ in the  proof of the part (2) above are also defined over 
	$\overline{\Q}$ (see \cite{Go}) and so the Galois conjugates 
	$C_{n}^{\sigma},p_{n}^{\sigma}$ and 	$F_{n}^{\sigma}$ of $C_{n},p_{n}$ and 	$F_{n}$ make perfect sense when $\sigma \in Gal( \overline{\Q}/\Q).$ 
	\end{proof}
	
	We can now prove that the isomorphism class of the group 
$\overline{\com(\Gamma)}$ is preserved under Galois action.

  \begin{theorem} \label{mainth} 
	 Let $C$ be an algebraic curve uniformized by  a Fuchsian group $\Gamma$.  Then the action of a Galois element $\sigma$ on the solenoid $\mathcal{H}_{C}$ induces the following isomorphisms
	\begin{enumerate}
	 \item \  \  \  $\overline{\com(\Gamma)} \simeq \overline{\com(\Gamma^{\sigma})}$ \ \ \ \ \ \  (isomorphism of topological groups).
	\item \ \ \  $\dfrac{\overline{\com(\Gamma)}}{\widehat{\Gamma}} \simeq 
	\dfrac{\overline{\com(\Gamma^{\sigma})}}{\widehat{\Gamma^{\sigma}}} $ \ \ \ \ \ \ \ \ \  (isomorphism of sets of co-sets).
\item \ \ \  $ \dfrac{ \com(\Gamma)}{\Gamma} \simeq \dfrac{ \com(\Gamma^{\sigma})}{\Gamma^{\sigma}} $ \ \ \ \ \  \ \ \ (isomorphism of sets of co-sets).	
	\end{enumerate}
    \end{theorem}
		\begin{proof}
		
	(1) That $\sigma$ induces an isomorphism between these two groups   follows from a combination 	of Theorem  \ref{Isom}
	and the first part of  Proposition \ref{autHSigma}. The topological side follows from the second part of  Proposition \ref{autHSigma} which shows that this isomorphism sends the subbasis of neighbourhoods $\{\widehat{\Gamma}_{n}\}_{n}$ into the subbasis of neighbourhoods $\{\widehat{\Gamma^{\sigma}_{n}}\}_{n}$.
	
	(2) follows from the second part of Proposition 
	\ref{autHSigma}.
	
	(3) follows from the previous statement and Proposition \ref{ComHat}, part 4.
\end{proof}

  From here we can easily obtain the following result due to Kazhdan(\cite{Kaz})
  \begin{corollary}\label{kaz}
  $\Gamma$ is arithmetic if and only if   $ \Gamma^{\sigma}$ is arithmetic.
   \end{corollary}
   \begin{proof}
  The result follows from the third part of Theorem \ref{mainth} together with  Margulis'  theorem.
 \end{proof}
 
 The general result that complex varieties uniformised by arithmetic groups are defined over number fields is atributed to Baily and Borel \cite{BB}
	(although in their paper only the fact that they are quasi-projective varieties seems to be explicitly stated and perhaps  reference to the work of Shimura is needed). 	
	In the case of dimension $1$  we are concerned with here, we can use Corollary \ref{kaz} to give
the following simple proof.

\begin{proposition} \label{DefinedOverNumberField}
	Let $C$ be  an algebraic curve uniformised by an arithmetic group $\Gamma$. Then $C$ is defined over a number field.
	\end{proposition} 
		\begin{proof}
	Let $\sigma \in Gal(\mathbb{C})$. By   Corollary \ref{kaz} the group $\Gamma^\sigma$ uniformising the algebraic curve $C^\sigma$ is also  arithmetic. Now, by a theorem of Takeuchi (\cite{Ta2}, Theorem 2.1)
there are only finitely
many arithmetic surface  groups of any given genus. Therefore, as $\sigma$ varies in $Gal(\mathbb{C})$,   only finitely many isomorphism classes of curves $C^\sigma$ are obtained. This means that $C$ is defined over a number field (see \cite{Go}).
	\end{proof}

 Unfortunately,   the bijections 
	$\sigma:\mathcal{H}_{C}  \rightarrow   \mathcal{H}_{C^{\sigma}}$ will   not  preserve the base leaf in general and so there is no a priori reason why the isomorphism $\overline{\com(\Gamma)} \simeq \overline{\com(\Gamma^{\sigma})}$ above should preserve any properties of the commensurator. However we will  show that:
\begin{enumerate}
\item The periods of $\com(\Gamma)$ are Galois invariant, and 
\item  In the non-arithmetic case  even the isomorphism class of the group 
$\com(\Gamma)$  is also   Galois invariant.
\end{enumerate}
  The proof of 1) will be carried out in Section \ref{AritGroups} which is devoted to arithmetic groups. The proof of 2) 
	 is the content of the next subsection.

 \subsection{The non-arithmetic case}

We next show that    for non-arithmetic groups the isomorphism class of the group 
$\com(\Gamma)$  is also   Galois invariant.
 In order to do that we need to recall the notion of \textit{semiregular  (or uniform)} covers. These are coverings of compact Riemann surfaces $f:C\to C'$ such that all points of $C$ within any given fibre have the same multiplicity. This is equivalent to saying
that $f:C\to C'$ corresponds  to the projection
 $\mathbb{H}/\Gamma \to \mathbb{H}/\Gamma'$ induced by an inclusion of $\Gamma$ in another Fuchsian group $\Gamma'$ (see e.g. \cite{GG}). Note that $\Gamma'$ will be the uniformising group of $C'$ only when $f$ is unramified. Otherwise it will be a \emph{Fuchsian group of signature} $(g'; m_1,\cdots,m_r)$ where $g'$ is the genus of $C'$ and $m_1,\cdots,m_r$ the multiplicities of the branching values $x'_1,\cdots,x'_r\in C'$ of the morphism $f.$ 
We will also need the following   straightforward implication of Margulis'  theorem.
 
\begin{lemma} \label{Highest}
\begin{enumerate}
\item Let $\Gamma$ be a non-arithmetic  Fuchsian surface group.  
 Then \newline $\com(\Gamma)$ is the largest  Fuchsian group containing $\Gamma.$ 
\item Let $C$ be an algebraic curve uniformised by a non-arithmetic Fuchsian group $\Gamma$. Then the uniform covering corresponding to the obvious projection $\mathbb{H}/\Gamma \to \mathbb{H}/\com(\Gamma)$ can be recognised as the  one  of highest degree among all uniform coverings   $f:C\to C'$ with source  $C$   and arbitrary target $C'.$
 \end{enumerate}
\end{lemma}
\begin{proof}
(1) By Margulis' theorem $\com(\Gamma)$  is a Fuchsian group.  Now, if a group
 $\Gamma_1$ containing $\Gamma $ is Fuchsian   then $\Gamma$ must have finite index in  $\Gamma_1.$ Thus, we have 
$\Gamma <\Gamma_1<\com(\Gamma_1)=\com(\Gamma).$ 
\newline
(2) Follows directly from (1). 
\end{proof}

 \begin{theorem} \label{ThNonArit} 
	 Let $C$ be an algebraic curve uniformized by  a non-arithmetic  Fuchsian group $\Gamma.$ Then  
	  $$\com(\Gamma) \simeq \com(\Gamma^{\sigma}) \hspace{1cm} \text{ (isomorphic as abstract groups)}$$  
    \end{theorem}
  \begin{proof}  
Let   $f:C\to C'$ be the uniform cover of highest degree given by Lemma \ref{Highest}.
	Then,  the Galois conjugate covering  $f^{\sigma}:C^{\sigma}\to C'^{\sigma}$ is also a uniform cover of highest degree. By Corollary \ref{kaz}  the group
	$\Gamma^{\sigma}$ is also non-arithmetic  and by  Lemma \ref{Highest} these two coverings correspond to the inclusions $\Gamma<\com(\Gamma)$ and 	$\Gamma^{\sigma}<\com(\Gamma^{\sigma})$ respectively. Moreover, 
	suppose that  $C'$  has genus $g'$ and $f$ has 
	$r$ branching values   with multiplicities  $m_1,\cdots,m_r$, then  the same holds for 
	$C'^{\sigma}$   and $f^{\sigma}$. This means that the groups $\com(\Gamma)$ and $\com(\Gamma^{\sigma})$ are both Fuchsian groups with the same signature  $(g'; m_1,\cdots,m_r)$ and therefore  isomorphic.
\end{proof}

\section{Arithmetic groups}  \label{AritGroups}
  In this final section we will study more closely the effect of 
Galois action 
on the invariant quaternion algebra $(k\Gamma,A\Gamma)$ of  an arithmetic Fuchsian surface  group $\Gamma$, that is,  an arithmetic Fuchsian   groups which uniformises an   algebraic curve $C\cong \HH/\Gamma$.
 In view of Proposition \ref{DefinedOverNumberField}
	we can --and we will--
	  replace the group $Gal(\mathbb{C}/\mathbb{Q})$ with the absolute Galois group
 $Gal(\overline{\mathbb{Q}}/\mathbb{Q}).$ 
We have  seen in Corollary\ref{kaz} that for any 
$\sigma\in \Gal(\overline{\mathbb{Q}}/\mathbb Q)$
the Galois conjugate curve $C^{\sigma}$ is isomorphic to a Riemann surface $\HH/\Gamma^{\sigma}$ where $\Gamma^{\sigma}$ is again arithmetic.
We would like to know how the quaternion algebras $(k\Gamma,A\Gamma)$ and $(k\Gamma^{\sigma}, A\Gamma^{\sigma})$ are related.

  This relation has been described  by Doi and Naganuma in \cite{doinaganuma} 
	(see also \cite{milnesuh}). In order to present their results we need to recall some facts from the theory of algebras over number fields. 

\subsection{Classification of quaternion algebras over number fields}\label{quatalgnf}

Let $k$ be a totally real number field. The embeddings $\tau\in \Hom(k,\mathbb C)(=\Hom(k,\mathbb R))$ are called \textit{infinite or archimedean places of $k$}. As already mentioned, the quaternion algebra $A=\left(\frac{a,b}{k}\right)$ is called \textit{unramified at the infinite place $\tau$} if
  $A^{\tau}\otimes_{k^{\tau}}\mathbb R$ 
is isomorphic to $M_2(\mathbb R)$, where 
$A^{\tau}=\left(\frac{\tau(a),\tau(b)}{k^{\tau}}\right)$.
 Otherwise (i.e. if   $A^{\tau}\otimes_{k^{\tau}}\mathbb R\cong \textbf{H}_{\mathbb R}$) $A$ is called \textit{ramified at $\tau$}. We denote by  $Ram_{\infty}A$ the set of all infinite places at which $A$ is ramified. Our condition on $A$ is that $A$ is unramifed exactly at $\tau=id$. Let $R_k$ denote the ring of integers in $k$ and let $\mathfrak p$ be a prime ideal in $R_k$. Then $\mathfrak p$ defines a non-archimedean absolute value on $k$, which is unique up to equivalence. The prime ideals in $R_k$ are also called \textit{finite or non-archimedean places of $k$}. Let $k_{\mathfrak p}$ denote the completion of $k$ with respect to this absolute value and $\tau_{\mathfrak p}: k\rightarrow k_{\mathfrak p}$ be the corresponding embedding. We can consider the quaternion algebra $$
A^{(\tau_{\mathfrak p})}=A\otimes_{k} k_{\mathfrak p}=\left( \frac{\tau_{\mathfrak p}(a),\tau_{\mathfrak p}(b)}{k_{\mathfrak p}}\right)
$$ 
over $k_{\mathfrak p}$. There are two possibilities for $A^{(\tau_{\mathfrak p})}$: either $A^{(\tau_{\mathfrak p})}\cong M_2(k_{\mathfrak p})$ or $A^{(\tau_{\mathfrak p})}$ is a division algebra which is then uniquely determined up to isomorphism. In the first case we say that \textit{$A$ unramified at the finite place $\mathfrak p$} whereas in the second case we call $A$ \textit{ramified at the finite place $\mathfrak p$}. 
By  $Ram_fA$ we will refer to the set of all finite places at which $A$ is ramified. We will further set $Ram(A)= Ram_{\infty}A\cup Ram_{f}A.$

Finally,  we note that given a number field $k$ and an automorphism  
$\sigma\in \Gal(\overline{\mathbb{Q}}/\mathbb Q)$, Galois conjugation  defines
an obvious  ring isomorphism between $R_k$ and $R_{k^{\sigma}}$ such that the rule 
$$\mathfrak p \to \mathfrak p^{\sigma}$$
provides a bijection between 
 prime ideals  of $R_k$ and  
prime ideals of  $ R_{k^{\sigma}}$ which induces a bijection between   
  $Ram_fA$ and  $(Ram_fA)^{\sigma}:=
	\{\mathfrak p^{\sigma}:\mathfrak p \in Ram_fA \}$. We observe in passing that $(Ram_fA)^{\sigma}=
	Ram_fA^{\sigma}.$

The following facts are well-known (see e.g. \cite{MR}, 
  Theorem 7.3.6, page 236):
\begin{enumerate}
\item $Ram(A)$ is a finite set of even cardinality
\item   Let $T$ be a finite set of  (archimedean and/or non-archimedean) places of $k$ with even cardinality, 
then there exists a quaternion algebra $A$ over $k$ such that $Ram(A)=T$. This quaternion algebra is uniquely determined up to isomorphism. 
\end{enumerate}

  \subsection{ Doi-Naganuma's theorem}

	Let $(k,A)=(k\Gamma,A\Gamma)$ be the quaternion algebra associated with an arithmetic surface group and let $\sigma\in \Gal(\overline{\mathbb{Q}}/\mathbb Q)$. 
	The structure of the pair 
	$(k\Gamma^{\sigma},A\Gamma^{\sigma})$ has been described by Doi and Naganuma in  \cite{doinaganuma} (see also \cite{milnesuh}).

	We first recall some definitions.  
	Let $\mathcal O$ be a maximal order in $A$. We define the group of totally positive units of $\mathcal O$ as $\mathcal O^+=A^+\cap \mathcal O^{\ast}$. Let $\mathfrak a$ be an ideal in $R_k.$
	The \textit{principal congruence subgroup of level $\mathfrak a$}
	is  the  group
	$$G_{\mathcal O}^+(\mathfrak a):=\{x\in \mathcal O^+\mid x-1\in \mathfrak a\mathcal O\}< GL^+_2(\mathbb R).$$
	 	
We will write $\Gamma_{\mathcal O}^+(\mathfrak a)=
 \widetilde{P}(	G_{\mathcal O}^+(\mathfrak a))$ where
$  \widetilde{P} : GL^+_2(\mathbb R)\rightarrow PSL_2(\mathbb R)$ stands for the obvious projection map. 
   
\begin{theorem}\label{mainII}(\cite{doinaganuma}) Let $\Gamma_{\mathcal O}^+(\mathfrak a)$ be a torsion-free principal congruence subgroup in $A$ and let  $C$ be an algebraic curve  analytically isomorphic to the Riemann surface  $\mathbb H/\Gamma_{\mathcal O}^+(\mathfrak a)$. Then, for any  $\sigma\in \Gal(\overline{\mathbb{Q}}/\mathbb Q)$, the conjugate curve $C^{\sigma}$ is analytically isomorphic to $ \mathbb H/\Gamma_{\mathcal O'}^+(\mathfrak a^{\sigma})$
where $\mathcal O'$ is a maximal order in the quaternion algebra $(k',A')$ determined, up to isomorphism, by the three following properties:
\begin{enumerate}
\item $k'=k^{\sigma}$.
\item  $Ram_f A'=(Ram_f A)^{\sigma}.$   
\item The only archimedean place at which $A'$ is unramified is the one corresponding to the identity on $k'$.
\end{enumerate}
	\end{theorem}
This implies the following

\begin{theorem}\label{mainI}
	Let $\Gamma$ be an arithmetic Fuchsian surface group and let $\sigma\in \Gal(\overline{\mathbb{Q}}/\mathbb Q)$.
	The associated quaternion algebra  $(k\Gamma^{\sigma},A\Gamma^{\sigma})$ is
determined, up to isomorphism, by the  following three properties:
	 \begin{enumerate}
\item $k\Gamma^{\sigma}=(k\Gamma)^{\sigma}$.
\item   $Ram_fA\Gamma^{\sigma}=(Ram_f A\Gamma)^{\sigma}.$ 
\item The only archimedean place at which $A\Gamma^{\sigma}$ is unramified is the one corresponding to the identity on $k\Gamma^{\sigma}$.
\end{enumerate}
 \end{theorem}

\begin{proof}
  Let 	  $\mathcal O$ be an order in $A\Gamma$ such that 
$\Gamma$ is commensurable to $P( \rho(\mathcal{O}^1)).$
We may assume that  $\mathcal O$ is a maximal order.

Let    $\Gamma_2=\Gamma_{\mathcal O}^+(\mathfrak a)$ be a torsion-free principal congruence subgroup as in Theorem \ref{mainII} and set $\Gamma_{12}=\Gamma\cap \Gamma_{2}$. 
Then the algebraic curve 
$C_{12}=\HH/\Gamma_{12}$ is simultaneously an unramified  cover of $C=\HH/\Gamma$ and $C_{2}=\HH/ \Gamma_2$, and therefore 
$C_{12}^{\sigma} $ is simultaneously an unramified cover of $C^{\sigma} $ and $C_{2}^{\sigma}$. Since quaternion algebras associated with arithmetic groups are only defined up to commensurability we  conclude that 
$A\Gamma=A\Gamma_{12}=A\Gamma_2$ 
and $A\Gamma^{\sigma}=A\Gamma_{12}^{\sigma}=A\Gamma_2^{\sigma}.$ Now apply Theorem \ref{mainII}.

\end{proof}


\begin{corollary} \label{sameA}The notation being as above, let $k$ be a normal extension of $\mathbb Q$.
Then $A\Gamma^{\sigma}=A\Gamma$ 
if and only if   $(Ram_f A\Gamma)^{\sigma}=Ram_f A\Gamma$. 
\end{corollary}


\begin{corollary}
Suppose that  $k\Gamma=\mathbb Q$. Then $A\Gamma^{\sigma}=A\Gamma.$ In particular 
$ \com(\Gamma^{\sigma})= \com(\Gamma).$
\end{corollary}  
 
\begin{example}
Let   $\Gamma_p$, $p$ odd prime, the family of groups 
constructed in  \ref{example}. As for two different prime numbers $p$ and $q$ the fields $\mathbb Q(\cos(2\pi/p))$ and $\mathbb Q(\cos(2\pi/q))$ are not Galois conjugate, Theorem \ref{mainI} implies that   no pair of surface subgroups $\Gamma_p'<\Gamma_p$ and $\Gamma_q'<\Gamma_q$ can be Galois conjugate.  
 \end{example}
  
  \subsection{Congruence property as a Galois invariant}\label{congruence}

		Let $ \widetilde{P} : GL^+_2(\mathbb R)\rightarrow PSL_2(\mathbb R)$ be the obvious projection. 
	A subgroup of  $G<GL^+_2(\mathbb R)$ 
	(resp. of  $\Gamma < PSL_2(\mathbb R)$) is called a 
	\textit{congruence subgroup} in $A$ if $G$ (resp.   
	$\widetilde{P}^{-1}(\Gamma)$)  contains some principal congruence 
	subgroup $G_{\mathcal O}^+(\mathfrak a)$ for some maximal order $\mathcal O$ and  some ideal $\mathfrak a$ of $\mathcal O$.

	\begin{proposition}\label{CongSubInvar}
Being uniformized by a congruent subgroup is a Galois invariant property.
\end{proposition}
	\begin{proof}
Let 	  $C$ be an algebraic curve isomorphic to 
$\HH/ \Gamma$ where $\widetilde{P}^{-1}(\Gamma)$ contains some principal congruence 
	subgroup $G_{\mathcal O}^+(\mathfrak a)$. 
	
	Let $\sigma\in Gal(\overline{\mathbb Q}/\mathbb Q)$.
	Then $C$ is covered by the curve $C_1:=\HH/ \Gamma_{\mathcal O}^+(\mathfrak a)$ and so $C^{\sigma}$ is covered by $C_1^{\sigma}.$ This implies that the uniformizing group of $C^{\sigma}$, which we have denoted   throughout $\Gamma^{\sigma}$, contains the uniformizing group of $C_1^{\sigma}$,
	 which by Theoerem \ref{mainII} is  a principal congruence subgroup of the form
	$\Gamma_{\mathcal O'}^+(\mathfrak a^{\sigma})$.
	This means that $\Gamma^{\sigma}$ is a congruence subgroup, as was to be seen.
	\end{proof}

	We mention another direct consequence of Theorem \ref{mainII}:
	
	\begin{corollary}\label{cormainII} Let $k$ be a normal extension of $\mathbb Q$ and $A$ a quaternion algebra over $k$ with following properties:
\begin{enumerate}
\item For any two maximal orders $\mathcal O$ and $\mathcal O'$ in $A$ there exists $x\in A$ such that $\mathcal O'=x^{-1}\mathcal O x$,
\item $Ram_f(A)$ is invariant under $Gal(k/\mathbb Q)$. 
\end{enumerate}
Let $\sigma\in Gal(\overline{\mathbb Q}/\mathbb Q)$ and $\mathfrak a$ be an ideal in $R_k$ such that $\mathfrak a^{\sigma}=\mathfrak a$. Then, the curve $C(\mathfrak a)=\mathbb H/\Gamma_{\mathcal O}^+(\mathfrak a)$ is isomorphic to its Galois conjugate curve $C(\mathfrak a)^{\sigma}$.
\end{corollary}
	%
	\begin{example} \label{example_cong}
 
 All the conditions of Corollary \ref{cormainII} are satisfied in the case where $k=\mathbb Q$ and $\Gamma_{\mathcal O}^+(\mathfrak a)$ is a principal congruence subgroup in a division quaternion algebra $A$ over $\mathbb Q$. This only reflects the fact that in this case the curves 
 $C(\mathfrak a)$ are defined over  $\mathbb Q$ 
(see e.g.  \cite{El1}, 2.3).
\end{example}

\begin{example}  
Consider the number field $k=\mathbb Q(\zeta_7+\zeta_7^{-1})$, where $\zeta_7=\exp(2\pi i/7)$, that is 
 $k=\mathbb Q(\cos \frac{2\pi}{7})$ is the totally real subfield of the cyclotomic field of $7$th roots of unity.
Let $A$ be the quaternion algebra over $k$ which is 
  unramified at the infinite place corresponding to the identity embedding and ramified at the two other infinite places and such that $Ram_f(A)=\emptyset$. This uniquely determines $A$ up to isomorphism (see \ref{quatalgnf}). The detailed study of this algebra carried out by Elkies in  the last section of \cite{El2} (see also \cite{El1}, 5.3)   allows us to illustrate the previous results in this case.

\
	
	Set 	$c=\zeta_7+\zeta_7^{-1}=2\cos \frac{2\pi}{7}$ so that $k=\mathbb Q(c)$ and $R_k=\mathbb Z [c]$. Then (\cite{El2}, 4.4):
	\begin{enumerate}[(1)]
	\item $A=\big(\dfrac{c,c}{k}\big)$ and there is an embedding $\rho:A\otimes_{k}\mathbb R \simeq 
 M_2(\mathbb R)$ given by the basis
	 $\vec{1}={\tiny \left( \begin{array}{cc}
1  & 0 \\
0  & 1\end{array} \right)}$,
$\vec{i}={\tiny \left( \begin{array}{cc}
\sqrt{c}  & 0 \\
0  & -\sqrt{c} \end{array} \right)}$,
$ \vec{j}={\tiny \left( \begin{array}{cc}
0  & \sqrt{c}\\
\sqrt{c} & 0 \end{array} \right)}$
and
$\vec{i}\vec{j}={\tiny \left( \begin{array}{cc}
0  & c\\
-c  &0 \end{array} \right)}$.
\item The ring 
$$\mathcal O=\mathbb Z [c][\vec{i},\vec{j'}]
=\mathbb Z [c]\vec{1}+\mathbb Z [c]\vec{i}+\mathbb Z [c]\vec{j'}+\mathbb Z [c]\vec{i}\vec{j'}, $$ 
where $\vec{j'}=\frac{1}{2}(1+ c\vec{i} +(c^2 + c +1)\vec{j} \ ),
$ is a maximal order in $A$ and the group $P(\mathcal O^1)$, which in this case agrees with $\tilde{P}(\mathcal O^+)  $, is isomorphic to the triangle   
   Fuchsian group $\Delta =\Delta(2,3,7)$  of signature $(2,3,7)$.
	\item There exists only one prime ideal $\mathfrak p_7$ above $p=7$. The Galois group $Gal(k/\mathbb Q)$ fixes $\mathfrak p_7$. The same is true for every prime $p\equiv \pm 2,\pm 3\bmod 7$: There exists only one prime ideal $\mathfrak p_p$ above $p$. Every automorphism of $k$ fixes $\mathfrak p_p$.
\item Above each prime   $p\equiv \pm 1\bmod 7$ there are three different prime ideals $\mathfrak p_{1,p},\mathfrak p_{2,p}$ and $\mathfrak p_{3,p}$. 
These three prime ideals form a Galois orbit; that is, after a possible renumeration $\mathfrak p_{2,p}=\mathfrak p_{1,p}^{\sigma}$ and $\mathfrak p_{3,p}=\mathfrak p_{1,p}^{\sigma^2}$, where $\sigma$ is the generator of the Galois group $Gal(k/\mathbb Q)$ (of order three). 
\end{enumerate}

 Now, let $\mathfrak q$ be a prime ideal in $R_k$ and denote by $\Gamma = \Gamma_{\mathcal O}^+(\mathfrak q)$ 
the corresponding principal congruence subgroup 
 so that $A=A\Gamma$.
Let $C(\mathfrak q) \cong \HH/\Gamma$ be the corresponding algebraic curve.  We claim that for every $\sigma\in Gal(\overline{\mathbb Q}/\mathbb Q)$ the curve  $C(\mathfrak q)^{\sigma}\cong \HH/\Gamma^{\sigma}$ agrees with the curve  
$C(\mathfrak q^{\sigma})$ uniformized by the principal congruence subgroup  $ \Gamma_{\mathcal O}^+(\mathfrak q^{\sigma})$ of $A$.  
This is because   Corollary \ref{sameA}
implies that  $A\Gamma^{\sigma}=A\Gamma=A$ 
and Theorem \ref{mainII} tells us that $\Gamma^{\sigma} = \Gamma_{\mathcal O'}^+(\mathfrak q^{\sigma})$
 where  $\mathcal O'$ is a maximal order in $A$. A computation of the narrow class  number $h_{\infty}$  associated with $A$ (see \cite{MR},p.221) gives  
$h_{\infty}=1$
 which means that 
  there is only one conjugacy class of  maximal orders, hence   up to conjugation by an element  $x\in A$ the order  $\mathcal O'$
  agrees with $\mathcal O.$
This allows us to draw the following conclusions:
\begin{enumerate}[(a)]
 \item If $\mathfrak q$ lies above a rational prime $p=7$ or $p\equiv \pm 2,\pm 3\bmod 7$ for every $\sigma\in Gal(\overline{\mathbb Q}/\mathbb Q)$ the curve $C(\mathfrak p)^{\sigma}$ is isomorphic to $C(\mathfrak p)$. This only reflects the fact that these curves are defined over $\mathbb Q$ (although see  Remark \ref{ModuliField} below).
 \item If $\mathfrak q=\mathfrak p_p$ is a prime ideal above a rational prime $p\equiv \pm 1\bmod 7$, then the $Gal(\overline{\mathbb Q}/\mathbb Q)$-orbit of $C(\mathfrak p_p)$ consists of three non-isomorphc curves $C(\mathfrak p_p)=C(\mathfrak p_{1,p})$, $C(\mathfrak p_{2,p})$,$C(\mathfrak p_{3,p}).$ Each of them remains invariant under the action of $Gal(\overline{\mathbb Q}/k)$ and, as above, this implies that they can be defined over the number field $k$.  
\end{enumerate}

One way to see that these three curves  are pairwise non-isomorphic or, equivalently, that the uniformizing groups
$\Gamma(\mathfrak p_{i,p})$ are not conjugate in $PSL_2(\mathbb R)$  is as follows: The triangle group $\Delta$ is known to be a maximal Fuchsian group. This implies that the
groups $\Gamma_i=\Gamma(\mathfrak p_{i,p})$ are not only 
normal in $\Delta$ but that, in fact, their normalizers 
$N(\Gamma_i)$ agree with $\Delta$. Now, if we had 
  $\Gamma_{i}=x\Gamma_{j}x^{-1} $ for some 
	$x\in PSL_2(\mathbb R)$, then    $\Gamma_{i}$ would be 
 contained in the maximal triangle groups  $\Delta$  and  $x \Delta x^{-1}$, hence we would have 
 $N(\Gamma_{i})=\Delta= x \Delta x^{-1}$. This would imply that the element $x$ lies in the normalizer of $\Delta$ which, by maximality, agrees with $\Delta$. This, in turn, would yield that $\Gamma_i=\Gamma_j$, a contradiction.

   These examples are geometrically interesting as they provide examples of Riemann surfaces of genus $>1$ with maximal number of automorphisms (see for instance \cite{JW}).
\end{example}

\begin{remark}\label{ModuliField}
   If $C$ is an arbitrary curve defined over a number field, the invariance of the isomorphism class of   $C$ under an absolute Galois group 
$ Gal(\overline{\mathbb Q}/k)$ 
only means that the so-called  \emph{field of moduli} of the curve is contained in  $k$.   Fortunately when, as in this case, the uniformizing group is normally  contained in a triangle group the field of moduli is also a field of definition \cite{Wo}.
\end{remark}


\subsection{The torsion of $\com(\Gamma)$ is Galois invariant}\label{torsiongaloisinvariant}

 We can use Theorem \ref{mainI} also to prove the equality $\mathcal P(\Gamma)=\mathcal P(\Gamma^{\sigma})$ between the sets of periods of the commensurator of $\Gamma$ and the commensurator of $\Gamma^{\sigma}$. In order to do so, we first recall (see \cite[Theorem 8.4.4]{MR}) that   the commensurator of $\Gamma$ in $\text{PSL}_2(\mathbb R)\cong \text{PGL}_2^+(\mathbb R)$ is $\tilde P(A^+)$.

 \subsubsection{Maclachlan's characterization of   torsion in 
 $\com(\Gamma).$ } \label{arithmeticgroups}

Following the work of Chinburg and Friedman \cite {CF}, C. Maclachlan (\cite{Mac1}, see also \cite {MR}, Lemma 12.5.6) showed the following result:  
 
\begin{proposition} \label{LemmaMR}
$\com(\Gamma)$ contains
 an element of order $m\geq 3$ if and only if the following properties hold: 
 \begin{enumerate}[i)]
\item $ \cos \frac{2\pi}{m}$ lies in the invariant trace field  
$k\Gamma.$
\item There is an embedding of $k\Gamma-$algebras $\varphi: k\Gamma(e^{2\pi i/m}) \hookrightarrow  A\Gamma.$
  \end{enumerate}
  In that case   $z=1+e^{2\pi i/m} \in A\Gamma$  provides such a finite order  element.
\end{proposition}


\

Actually,  in  Lemma 12.5.6  of \cite{MR} this result  is stated   for $P(A^{\star})$ instead of 
$\com(\Gamma)=P(A^{+})$ but the result holds just as well for $P(A^{+})$ because the element $z=1+e^{2\pi i/m}$ lies in  $P(A^{+})$, since its image in 
$A\Gamma$, 
$\varphi(z)={\tiny \left( \begin{array}{cc}
1+ \cos \frac{2\pi}{q}  & \sin \frac{2\pi}{q} \\
    &  \\
-\sin \frac{2\pi}{q}  & 1+ \cos \frac{2\pi}{q} \end{array} \right)}$,  has positive determinant. Note that with respect to other embeddings the positivity of the norm is automatically satisfied, as the norm form in Hamiltonian quaternions is  a positive definite quadratic form.  

\begin{theorem}  \label{periods}
 Let $\mathcal{P}(\Gamma)\subset \mathbb{N}$ denote the set of orders
		(or periods)  of finite order  elements of   $\com(\Gamma)$. Then  $\mathcal{P}(\Gamma) = 
	\mathcal{P}(\Gamma^{\sigma})$, for any $\sigma \in Gal( \C/\Q)$.  
\end{theorem}
 \begin{proof}
 $P(A\Gamma^{+})$ always contains an element of order $2$.  So, we need to prove that if
($k\Gamma, A\Gamma)$ satisfies the  conditions i) and ii) of the above Proposition \ref{LemmaMR} then so does
$(k\Gamma^{\sigma}, A\Gamma^{\sigma})$.

This clearly holds for i) since by the 
Doi-Naganuma's theorem 
$k\Gamma^{\sigma}=(k\Gamma)^{\sigma}$. 

   In order to prove that this is also the case for ii) we first recall a criterion due to Brauer, Hasse and Noether for embedding quadratic field extensions  into   quaternion algebras  (see\cite[Theorem 7.3.3]{MR} and \cite{CF}): 
\begin{itemize}
\item A quadratic field extension $K$ of a number field $k$ can be embedded into a quaternion algebra $A$ over $k$ if and only if every archimedean or non-archimedean place of $k$ ramified in $A$ either ramifies or remains prime in $K$. 
\end{itemize} 

This applied to our situation tells us that what we need to prove is that if this condition is satisfied for the the quaternion algebra $(k,A)=(k\Gamma,A\Gamma)$ and 
the quadratic extension $K=k\Gamma(e^{2\pi i/m})$ of  $k\Gamma$ then so is for 
 the quaternion algebra $(k\Gamma^{\sigma} ,A\Gamma^{\sigma})$ and 
the quadratic extension $K^{\sigma}=k\Gamma^{\sigma}(e^{2\pi i/m})$ of  $k\Gamma^{\sigma}.$

Let us deal first with the non-archimedean places. If  $\mathfrak p \in Ram_f(A\Gamma)$ ramifies (resp. remains prime)  in $K=k\Gamma(e^{2\pi i/m})$ then    so does $\mathfrak p^{\sigma} \in Ram_f(A\Gamma)^{\sigma} = Ram_f(A\Gamma^{\sigma})$ in $K^{\sigma}.$   
This can be seen as follows.
Let $\mathfrak P$ be a prime ideal in $R_K$ lying over $\mathfrak p$ and let $\tau: K\to K$ be the non-trivial $k$-automorphism of $K$.
Then $\mathfrak p$ is ramified or remains prime in $K$ if and only if $\mathfrak P^{\tau}=\mathfrak P$. 
The non-trivial $ k^{\sigma}$-automorphism of $K^{\sigma}$ is then $\sigma^{-1}\tau\sigma$ and, clearly, $(\mathfrak P^{\sigma})^{\sigma^{-1}\tau\sigma}=\mathfrak P^{\sigma}$. Hence $\mathfrak p^{\sigma}$ also remains prime or is ramified
in $K^{\sigma}.$

As for the archimedean places we recall that
an archimedean place $v$ of a totally real number field $k$
is ramified in a Galois field extension $K/k$ if the embedding $v:k \to \mathbb{C}$ extends to an embedding $w:K\to \mathbb C$ whose image is not a subfield of the reals. In our case the fields $K=k\Gamma(e^{2\pi i/m})$  and  $K^{\sigma}=(k\Gamma)^{\sigma}(e^{2\pi i/m})$ 
are totally imaginary extensions of the totally real fields
$k\Gamma$ and  $k\Gamma^{\sigma}$
 so the ramification condition obviously holds.
  The proof is done. 
\end{proof}

When $m$ is odd these two conditions in Proposition  \ref{LemmaMR} can be formulated in the following simpler manner

\begin{proposition} \label{InvarAlgebr}
Let   $m$ be an odd number. Then $\com(\Gamma)$ contains an element of order $m$ if and only if $A\Gamma$ contains
a square root of $-\sin^2 \frac{2\pi}{m}.$
In particular $A\Gamma$ contains
a square root of $-\sin^2  \frac{2\pi}{m}$ if and only if $A\Gamma^{\sigma}$ does.
 \end{proposition}
\begin{proof}
 We must show that, for $m$ odd, the above conditions i) and ii) are equivalent to the existence of an element  $X\in A\Gamma$ such that $X^2=-\sin^2 \frac{2\pi}{m}
\in k\Gamma\subset A\Gamma$.

In one direction this is easy. If the first condition is satisfied then  first of all $ \sin^2 \frac{2\pi}{m}=1-\cos^2 \frac{2\pi}{m}\in k\Gamma$  and, moreover, $k\Gamma(e^{2\pi i/m})=k\Gamma(i \sin \frac{2\pi}{m})$. Now, if in addition, there is an embedding 
$\varphi: k\Gamma(i \sin \frac{2\pi}{m}) \to  A\Gamma$ 
then $X=\varphi(i \sin \frac{2\pi}{m})$ will provide the required square root.

Conversely, if such $X$ exists then $ \cos\frac{4\pi}{m}= 1-2\sin^2 \frac{2\pi}{m} \in k\Gamma.$ 
This means that $k\Gamma$ contains the field 
$\mathbb{Q}(\cos\frac{4\pi}{m})$. But this   is precisely 
the Galois subextension of the field $\mathbb{Q}(e^{4\pi i/m})$ fixed by complex conjugation. Now, $m$ being odd, 
$\mathbb{Q}(e^{4\pi i/m})=\mathbb{Q}(e^{2\pi i/m})$, hence 
$\mathbb{Q}(\cos\frac{4\pi}{m})=\mathbb{Q}(\cos\frac{2\pi}{m})$ and so the first condition  is satisfied. This in turn implies that $k\Gamma(e^{2\pi i/m})=k\Gamma(i \sin \frac{2\pi}{m})$, as before, and now simply  sending $i \sin \frac{2\pi}{m}$ to $X$ gives an embedding of $k\Gamma(e^{2\pi i/m})$ in $A\Gamma$, which is the second condition.
\end{proof}

\begin{example} \label{Noncommens}
Let   $\Gamma_p$, $p$ odd prime, the family of groups 
constructed in  \ref{example}. Then
 $\com(\Gamma_p)$ 
contains  an element of odd prime order $q$ if and only if $q=p.$ 
%
%
%
In particular, the groups  $\com(\Gamma_p)$ and $\com(\Gamma_p)$
are not isomorphic,   if $p\neq q.$
  \end{example} 
  These claims can be settled as follows: 

 By construction (see  \ref{example}) $k\Gamma_p=\mathbb{Q}(\sin\frac{2\pi}{p}).$ Thus, by Proposition  \ref{InvarAlgebr}, in order to prove that  $\com(\Gamma_p)$ 
has an element of order $p$
it is enough to observe that $A\Gamma^p$ contains  some root of $-1$, namely $ X= {\tiny
\left( \begin{array}{cc}
0 & 1 \\
-1 & 0 \end{array} \right)}.$  

 \vspace{0.2cm}
Now suppose that $q\neq p$. Then $\sin^2 \frac{2\pi}{q}  =1-\cos^2 \frac{2\pi}{q}$ does not lie in  $\mathbb{Q}(\sin \frac{2\pi}{p} )= \mathbb{Q}(\cos \frac{2\pi}{4p})$ (see Lemma \ref{lemma1}). This is because
otherwise  $ \cos^2 \frac{2\pi}{q} $ would lie in the intersection field 
$ \mathbb{Q}(e^{2\pi i/q}) \cap \mathbb{Q}(e^{2\pi i/4p})$ which is  
equal to $\mathbb{Q}$ since $q$ and $4p$ are co-prime
(see e.g. \cite{Wa}, Proposition 2.4). Now from Proposition \ref{InvarAlgebr} we   deduce that  $\com(\Gamma_p)$ cannot contain elements of order $q.$

\subsection{Final result}\label{Finalresult}

		Although, there are only 
		finitely many   arithmetic surface groups of given  genus    (\cite{Ta2}, Theorem 2.1) 	the  group $Gal(\overline{\mathbb{Q}}/\mathbb{Q})$ is going to act faithfully on them. Our final theorem records this fact and
collects the main invariants we have found for this action.
 
	\begin{theorem} \label{AbsGalGroup}
	$Gal(\overline{\mathbb{Q}}/\mathbb{Q})$ acts faithfully on the set of  isomorphy classes of arithmetic surface groups 
	  $\Gamma$
	and this action has the following invariants:
	 \begin{enumerate}
	\item The isomorphism class of the group $\overline{\com(\Gamma)}.$ 
	\item The Galois conjugacy class of $k\Gamma$. (In fact $k\Gamma^{\sigma}=(k\Gamma)^{\sigma}$ for any $\sigma \in Gal(\overline{\mathbb{Q}}/\mathbb{Q})$).
	\item The set $\mathcal{P}(\Gamma)$ of periods of 
	$\com(\Gamma).$ 
	\item The solvability of the quadratic equations 
	\  \  $X^2+\sin^2 \frac{2\pi}{2k+1},  \ k \in \mathbb{N},$
	\newline
	in the invariant quaternion algebra $A\Gamma$. 
	\item The property of being a conguence subgroup.
  \end{enumerate}
	\end{theorem}
	\begin{proof}
	
	That the action transforms arithmetically uniformised  curves into themselves is the content of Corollary \ref{kaz}. Faithfulness  is a consequence of the result proved in \cite{GJ} that the action is faithful on the set of curves uniformised by subgroups of any given triangle group together with the fact that there are plenty of triangle groups which are arithmetic \cite{Ta1}.
	
	As for the three listed invariants, 1)   is the first part of Theorem \ref{mainth},  2) is the first part of Theorem \ref{mainI}, 
	3)  is Theorem \ref{periods},    4)  is Proposition  \ref{InvarAlgebr} and 5) is Proposition \ref{CongSubInvar}.
	\end{proof}

 \textbf{Acknowledgements.} The authors   would like to thank 
 Adri\'an Ubis for the explicit choice of the constant $b_p$ in \ref{example}, J\"urgen Wolfart   for many valuable suggestions and Andrei Jaikin-Zapirain  for his helpful comments at an early stage of this paper.

\end{document}